\documentclass{amsart}

\usepackage{amssymb}
\usepackage{amsmath}
\usepackage{amsthm}
\usepackage{url}
\usepackage{enumerate}
\usepackage{graphicx}

\usepackage{hyperref}
\hypersetup{
colorlinks  =   true,
linkcolor   =   blue,
filecolor   =   blue,
urlcolor    =   blue,
citecolor   =   blue,
}

\newcommand{\erase}[1]{}

\theoremstyle{theorem}
\newtheorem{theorem}{Theorem}[section]
\newtheorem{lemma}[theorem]{Lemma}
\newtheorem{proposition}[theorem]{Proposition}
\newtheorem{corollary}[theorem]{Corollary}

\theoremstyle{definition}

\newtheorem{definition}[theorem]{Definition}
\newtheorem{remark}[theorem]{\it Remark}
\newtheorem{example}[theorem]{Example}

\numberwithin{equation}{section}
\numberwithin{table}{section}
\numberwithin{figure}{section}
\renewcommand{\qed}{\hfill {$\Box$}}

\newcommand{\C}{\mathord{\mathbb C}}

\newcommand{\F}{\mathord{\mathbb F}}

\renewcommand{\P}{\mathord{\mathbb  P}}
\newcommand{\Q}{\mathord{\mathbb  Q}}
\newcommand{\R}{\mathord{\mathbb R}}

\newcommand{\Z}{\mathord{\mathbb Z}}

\newcommand{\AAA}{\mathord{\mathcal A}}

\newcommand{\CCC}{\mathord{\mathcal C}}

\newcommand{\FFF}{\mathord{\mathcal F}}

\newcommand{\III}{\mathord{\mathcal I}}

\newcommand{\LLL}{\mathord{\mathcal L}}
\newcommand{\MMM}{\mathord{\mathcal M}}

\newcommand{\OOO}{\mathord{\mathcal O}}
\newcommand{\PPP}{\mathord{\mathcal P}}

\newcommand{\SSS}{\mathord{\mathcal S}}
\newcommand{\TTT}{\mathord{\mathcal T}}
\newcommand{\UUU}{\mathord{\mathcal U}}
\newcommand{\VVV}{\mathord{\mathcal V}}

\newcommand{\XXX}{\mathord{\mathcal X}}

\newcommand{\ZZZ}{\mathord{\mathcal Z}}

\newcommand{\SSSS}{\mathord{\mathfrak S}}

\newcommand{\maprightsp}[1]{\; \smash{\mathop{\; \longrightarrow \; }\limits\sp{#1}}\; }

\newcommand{\mapdown}{\phantom{\Big\downarrow}\hskip -8pt \downarrow}

\newcommand{\mapdownsurj}{
\hbox{$\bigm\downarrow$}
\llap{\hbox{\raise 2pt\hbox{$\bigm\downarrow$}}}%
\vstrechmapdown
}

\newcommand{\mapupsurj}{
\hbox{$\bigm\uparrow$}
\llap{\hbox{\raise 2pt\hbox{$\bigm\uparrow$}}}%
\vstrechmapup
}

\newcommand{\inj}{\hookrightarrow}

\newcommand{\isom}{\raise -1pt \hbox{\,$\xrightarrow{\sim\,}$\,}}

\newcommand{\set}[2]{\{\, {#1} \,|  \,{#2}\,   \}}

\newcommand{\sethd}[3]{\left\{\;\;  {#1}\;\; \left|\;\;  \vcenter{\hbox{\parbox{#2}{#3}}}\;\;  \right. \right\}}

\newcommand{\bigset}[2]{\left\{\; {#1} \; \left\vert \; {#2} \;  \right.\right \}}

\newcommand{\tensor}{\otimes}

\newcommand{\sprime}{\sp\prime}

\newcommand{\spprime}{\sp{\prime\prime}}

\newcommand{\sptimes}{\sp{\times}}
\newcommand{\sperp}{\sp{\perp}}

\newcommand{\dual}{\sp{\vee}}

\newcommand{\inv}{\sp{-1}}

\newcommand{\bdr}{\partial\,}

\newcommand{\PGL}{\mathord{\mathrm{PGL}}}

\newcommand{\OG}{\mathord{\mathrm{O}}}

\newcommand{\Image}{\operatorname{\mathrm{Im}}\nolimits}

\newcommand{\Gal}{\operatorname{\mathrm{Gal}}\nolimits}

\newcommand{\pione}{\pi_1}

\newcommand{\intf}[1]{\langle #1 \rangle}

\newcommand{\barL}{\overline{L}}
\newcommand{\tilL}{\widetilde{L}}

\newcommand{\barC}{\overline{C}}

\newcommand{\barl}{\bar{l}}
\newcommand{\barc}{\bar{c}}

\newcommand{\Rats}{\mathord{\mathit{Rats}}}
\newcommand{\Ells}{\mathord{\mathit{Ells}}}
\newcommand{\GO}{\mathord{\mathrm{GO}}}
\newcommand{\disjsevenset}{L^{\{7\}}}
\newcommand{\disjsevenlist}{L^{[7]}}
\newcommand{\disjsevenlistfam}{\LLL^{[7]}}
\newcommand{\bl}{\mathord{\mathrm{bl}}}
\newcommand{\LL}{\mathord{\mathbf{L}}}
\newcommand{\PP}{\mathord{\mathbf{P}}}
\newcommand{\PPseven}{\PP^{[7]}}
\newcommand{\gens}[1]{\langle #1\rangle}
\newcommand{\diag}{\mathord{\mathrm{diag}}}
\newcommand{\SX}{S\hskip -.9ptX}
\newcommand{\SY}{SY}

\newcommand{\SSSY}{\mathrm{\mathcal{SY}}}
\newcommand{\barSY}{\overline{SY}}
\newcommand{\barSSSY}{\overline{\SSSY}}
\newcommand{\tilh}{\tilde{h}}
\newcommand{\till}{\tilde{l}}
\newcommand{\tilc}{\tilde{c}}
\newcommand{\tilv}{\tilde{v}}
\newcommand{\tilgamma}{\tilde{\gamma}}

\newcommand{\barN}{\overline{N}}
\newcommand{\barR}{\overline{R}}
\newcommand{\barF}{\overline{F}}
\newcommand{\barPhi}{\overline{\Phi}}
\newcommand{\barFFF}{\overline{\FFF}}
\newcommand{\barLLL}{\overline{\LLL}}
\newcommand{\barTTT}{\overline{\TTT}}
\newcommand{\barCCC}{\overline{\CCC}}
\newcommand{\barNsprime}{\overline{N}^{\,\lower 1pt \hbox{\scriptsize $\prime$}}}
\newcommand{\barPPP}{\overline{\PPP}}
\newcommand{\barSigma}{\overline{\Sigma}}
\newcommand{\typeI}{\mathord{\mathrm{I}}}
\newcommand{\typeII}{\mathord{\mathrm{II}}}
\newcommand{\typeIII}{\mathord{\mathrm{III}}}
\newcommand{\intfX}[1]{\intf{#1}_X}
\newcommand{\intfY}[1]{\intf{#1}_Y}

\newcommand{\spmn}{^{(m,n)}}
\newcommand{\PhiUUU}{\Phi_{\hskip .3pt\UUU}}
\newcommand{\barPhiUUU}{\barPhi_{\hskip .3pt\UUU}}
\newcommand{\iotaUUU}{\iota_{\hskip .3pt\UUU}}
\newcommand{\Qtype}{\mathord{\mathcal{Q}}}
\newcommand{\Qmn}{\Qtype^{(m,n)}}
\newcommand{\UUUR}{\UUU_{\,\R}}
\newcommand{\UUURf}{\UUU_{\,\R, 4}}

\newcommand{\wtAtwo}{\setlength\unitlength{.2truecm}%
\begin{picture}(4.7,1.7)(0,0)%
\put(.9,.5){\circle{.7}}%
\put(3.9,.5){\circle{.7}}%
\put(1.25,.66){\line(1,0){2.3}}%
\put(1.25,.34){\line(1,0){2.3}}%
\end{picture}}%

\newcommand{\Pic}{\mathord{\mathrm{Pic}}}

\begin{document}
\title[Quartic curves]
{Zariski multiples associated with quartic curves}

\author[I. Shimada]{Ichiro Shimada}
\address{Department of Mathematics, Graduate School of Science, Hiroshima University,
1-3-1 Kagamiyama, Higashi-Hiroshima, 739-8526 JAPAN}
\email{ichiro-shimada@hiroshima-u.ac.jp}

\begin{abstract}
We investigate Zariski multiples  of 
plane curves $Z_1, \dots, Z_N$ 
such that each $Z_i$ is a union of a smooth quartic curve, some of its bitangents,
and  some of its $4$-tangent conics.
We show that,
for plane curves of this type,
the deformation types are equal to
the homeomorphism  types,
and that the number of deformation types grows 
as $O(d^{\,62})$ when the degree $d$ of the plane curves tends to infinity.
\end{abstract}

\thanks{This work was supported by JSPS KAKENHI Grant Number~20H01798~and~20K20879.}


\maketitle
%
\section{Introduction}\label{sec:Introduction}
By a \emph{plane curve},
we mean a reduced, possibly reducible, complex projective plane curve. 
We say that two plane curves $C$ and $C\sprime$ have
the same \emph{combinatorial type}
if there exist tubular neighbourhoods  $T(C)$ of $C$ 
and $T(C\sprime)$ of $C\sprime$  
such that $(T(C), C)$ and $(T(C\sprime), C\sprime)$ are homeomorphic,
whereas we say that $C$ and $C\sprime$ have
the same \emph{homeomorphism type}
if $(\P^2, C)$ and $(\P^2, C\sprime)$ are homeomorphic.
\par
A \emph{Zariski $N$-ple} is a set of plane curves
$\{C_1, \dots, C_N\}$ 
such that 
the curves $C_i$ have the same combinatorial type,
but their homeomorphism types are pairwise different.
The notion of  Zariski $N$-ples was introduced by Artal-Bartolo~\cite{Artal1994}
in reviving a classical example of $6$-cuspidal curves of degree $6$ due to Zariski.
Since then, 
this notion has been studied by many people from various points of view.
Some of the tools that have been used  in this investigation are:
Alexander polynomials, characteristic varieties,   fundamental groups of complements, 
topological invariants of  branched coverings, and so on. 
See the survey paper~\cite{Survey2008}.
Recently,  
Bannai et al.~\cite{Bannai2016, BannaiOhno2020, BannaiEtAl2019} have investigated Zariski $N$-ples
such that each member is a union of a smooth quartic curve and some of its bitangents.
\par
In this paper,
we introduce  $4$-tangent conics of a smooth quartic curve,
and consider 
Zariski $N$-ples $Z_1, \dots, Z_N$ 
such that each $Z_i$ is a union of a  smooth quartic curve, some of its bitangents,
and  some of its $4$-tangent conics. 
\par
Let  $Q$ be a smooth quartic curve.
A \emph{bitangent} of $Q$ is a line whose intersection multiplicity 
with $Q$ is even at each intersection point.
It is well known that every smooth quartic curve has 
exactly $28$ bitangents.
We say that a bitangent $\barl$ of $Q$ is \emph{ordinary}
if $\barl$ is tangent to $Q$ at distinct $2$ points. 
A smooth conic $\barc\subset \P^2$ is called a \emph{$4$-tangent conic of $Q$} 
if $\barc$ is tangent to $Q$ at $4$ distinct points.
Every smooth quartic curve has $63$ one-dimensional connected families of $4$-tangent conics
(see Theorem~\ref{thm:barCCC}).
\begin{definition}\label{def:Qmn}
Let  $m$ and $n$ be non-negative integers such that $m\le 28$.
We say that  a plane curve $Z$ is a \emph{$\Qmn$-curve} 
if $Z$ is of the form 
\begin{equation}\label{eq:Z}
Z=Q +\barl_1+\cdots +\barl_m+\barc_1+\cdots +\barc_n,
\end{equation}
where $Q$ is a smooth quartic curve, 
$\barl_1, \dots, \barl_m$ are distinct bitangents of  $Q$, and 
$\barc_1, \dots, \barc_n$ are distinct $4$-tangent conics of $Q$,
and they satisfy that 
\begin{enumerate}[{\rm (i)}]
\item  the bitangents $\barl_1, \dots, \barl_m$ are ordinary, 
\item the intersection of any  three of $Q, \barl_1, \dots, \barl_m, \barc_1, \dots, \barc_n$ is empty,
and 
\item the intersection of any  two of $\barl_1, \dots, \barl_m, \barc_1, \dots, \barc_n$ is transverse.
\end{enumerate}
\end{definition}
It is obvious that any two  $\Qmn$-curves have 
the same combinatorial type.
We construct a non-singular variety $\ZZZ\spmn$ 
parameterizing all $\Qmn$-curves in Section~\ref{sec:ZZZ}.
\begin{definition}\label{def:deftype}
We say that two $\Qmn$-curves
have the same \emph{deformation type}
if they belong to the same connected component of
the parameter space $\ZZZ\spmn$. 
\end{definition}
It is obvious that  $\Qmn$-curves of  the same deformation type
have the same homeomorphism type.
Our main results are the following:
\begin{theorem}\label{thm:main}
If two $\Qmn$-curves 
have the same homeomorphism type,
then they have the same deformation type.
\end{theorem}
We put
\begin{equation}\label{eq:dmn}
d\spmn:=\binom{28}{m}\cdot \binom{n+62}{62},
\end{equation}
which grows  as  $O( n^{62})$ when  $n\to \infty$.
\begin{theorem}\label{thm:main2}
The number $N\spmn$ of
deformation types of $\Qmn$-curves
satisfies
\begin{equation}\label{eq:estimate}
d\spmn/1451520 \le N\spmn \le d\spmn.
\end{equation}
\end{theorem}
The main ingredient of the proof of these results is 
the monodromy argument of Harris~\cite{Harris1979} (see Theorem~\ref{thm:monodromy}).
This argument converts the problem of enumerating deformation types of $\Qmn$-curves 
to an easy combinatorial problem of counting orbits of an action of 
the Weyl group $W(E_7)$ 
on a certain finite set. 
In Tables~\ref{table:Nm0},~\ref{table:N0n},~\ref{table:Nmn}, we give a list of $N\spmn$
for some $(m, n)$.
See~Section~\ref{subsec:computationalexample} for more detail.
\begin{table}
{\small
\[
\begin{array}{c|cccccccccccccc}
m & 1 & 2 & 3 & 4 & 5 & 6 & 7 &8 &  9 & 10 & 11 & 12 & 13 & 14 \\
\hline 
N &  1 &1 & 2 &3 &5 & 10&16 &23& 37& 54&70& 90& 101& 103 \\
G & 1 &1 & 2 &3 &5 & 9&16 &23& 37& 54&70& 90& 101& 103
\end{array}
\]
\caption{$N^{(m, 0)}=N^{(28-m, 0)}$}\label{table:Nm0}
\vskip -3pt 
\[
\begin{array}{c|cccccccccc}
n & 1 & 2 & 3 & 4 & 5 & 6 & 7 &8  &9 &10\\
\hline 
N &  1& 3& 9& 30&112&501& 2483& 13791 &81404 &490750\\
G & 1 & 3 & 7 & 22 & 71 & 306 & 1585 & 9831 & 64790 & 425252
\end{array}
\]
}
\caption{$N^{(0, n)}$}\label{table:N0n}
%
\vskip -3pt 
{\small
\[
\begin{array}{c|cccccccccc}
(m, n) & (1,1) & (1,2) & (2,1) & (1,3) &(2,2) &(3,1)& (1,4) &(2,3) &(3,2) &(4,1)  \\
\hline 
N &  2& 8& 4 &33 & 23 & 9 &162 & 132 &66 & 20 \\
G &2 & 8 & 3 & 30 & 17 & 8 & 140 & 95 & 57 & 17 
\end{array}
\]
\vskip 2pt
\[
\begin{array}{c|ccccccccccc}
(m, n) & (1,5) &(2,4) &(3,3)&(4,2) &(5,1)  &(1, 6)  &(2, 5) &(3,4) &(4,3) &(5, 2) &(6,1)\\
\hline 
N &  901  &889 &508 &190 &45 &5674 &6503 &4348 &1854  &531 &103\\
G &  753& 670 & 430 & 164 & 42 & 4829 & 5259 & 3812  & 1649  & 501&96
\end{array}
\]
}
\caption{$N^{(m, n)}$}\label{table:Nmn}
\end{table}
\par
To each $\Qmn$-curve $Z$,
we associate a discrete invariant $g(Z)$,
which we call an \emph{intersection graph}.
This invariant is  similar to the \emph{splitting graph} defined in~\cite{Shirane2019}.
Note that  each of the bitangents $\barl_1, \dots, \barl_m$ and a  $4$-tangent conics 
$\barc_1, \dots, \barc_n$ of the smooth quartic curve $Q\subset Z$
splits by the double covering $Y\to \P^2$ branched along  $Q$.
This data $g(Z)$ describes 
how the irreducible components of these pull-backs intersect on $Y$.  
See Section~\ref{sec:intgraph} for the precise definition.
In Tables,
we also present 
the number $G$ of non-isomorphic intersection graphs
obtained from $\Qmn$-curves.
When $n=0$, the intersection graph $g(Z)$ is the two-graph studied in~\cite{BannaiOhno2020},
in which Bannai and Ohno studied  $\Qtype^{(m, 0)}$-curves  for $m\le 6$,
and enumerated their homeomorphism types that can be distinguished by the two-graphs.
See Sections~\ref{subsec:m6n0}~and~\ref{subsec:n0} for the details. 
\par
The $4$-tangent conics of a smooth quartic curve $Q$
are related to the $2$-torsion points 
of the Jacobian of $Q$  (see Remark~\ref{rem:Pic}).
A similar idea applied to plane cubic curves
enabled us to  construct in~\cite{Shimada2003} 
certain equisingular families of plane curves with  many connected components.
In~\cite{Shimada2003}, it was also shown that 
these connected components cannot be distinguished by the fundamental groups of the complements,
because they are all ablelian. 
Then it was shown in~\cite{GBS2017} and~\cite{Shirane2017} 
that  the homeomorphism types of distinct connected components of  these families 
can be distinguished by the invariant called \emph{linking numbers}.
\par
The embedding topology of 
 reducible plane curves whose irreducible components are tangent to each other  was also  investigated by 
 Artal Bartolo,  Cogolludo,  and Tokunaga in~\cite{ABCT2008} from the view point of dihedral covering of the plane
 branched along the curve.
In this case, the complement can have  a non-abelian fundamental group.
See~\cite[Corollary 1]{ABCT2008}.  
 \par
It would be  an interesting problem to study the fundamental groups of the complements of $\Qtype^{(m, n)}$-curves,
and their related invariants such as linking numbers and/or  (non-)existence of finite coverings of the plane with prescribed Galois groups.
\par
Via the cyclic covering of the plane of degree $4$ branched along a smooth quartic curve,
the geometry of $\Qtype^{(m, n)}$-curves is  related to the geometry of $K3$ surfaces.
By considering 
the double  covering  branched along a singular sextic curve and employing Torelli theorem for complex $K3$ surfaces,
we have investigated in~\cite{Shimada2010}  Zariski $N$-ples of plane curves of degree $6$ with only simple singularities.
We expect that a similar idea can be applied to Zariski $N$-ples
associated with  singular quartic curves.
\par
This paper is organized as follows.
In Section~\ref{sec:split},
we introduce cyclic  coverings $X_u\to Y_u\to \P^2$
branched over a smooth  quartic curve $Q_u$,
and fix some notation.
In Section~\ref{sec:monodromy},
we recall  the  result of Harris~\cite{Harris1979}.
In Section~\ref{sec:familyofconics},
we construct the family of $4$-tangent conics.
In Section~\ref{sec:ZZZ},
we construct the space $\ZZZ\spmn$ 
parameterizing all $\Qmn$-curves,
and prove Theorems~\ref{thm:main}~and~\ref{thm:main2}.
In Section~\ref{sec:familyofconics2},
we further study the family of $4$-tangent conics in detail.
The geometry of the $K3$ surface $X_u$ is closely investigated.
In Section~\ref{sec:config},
we study the configurations of lifts in $Y_u$ of 
bitangents and $4$-tangent conics,
and we define the intersection graph $g(Z)$ in 
Section~\ref{sec:intgraph}.
In Section~\ref{sec:examples},
we examine some examples for small $m$ and $n$.
\par
\medskip
{\bf Acknowledgement.}
Thanks are due to  Professor Shinzo Bannai and Professor Taketo Shirane for discussions and comments.
The author also thanks the anonymous referee for his/her valuable comments on the first version of this paper.
\section{Coverings of $\P^2$}\label{sec:split}
For a positive integer $d$,
we put
\[
\Gamma(d):=H^0(\P^2, \OOO(d)).
\]
Let $\UUU$ denote the space of smooth quartic curves,
which is a Zariski open subset of $\P_*(\Gamma(4))$.
Let $u$ be a point of $\UUU$.
We denote 
by $Q_u\subset \P^2$ the smooth quartic curve corresponding to the point $u$.
We consider the following  branched coverings:
\[
\gamma_u \colon X_u\maprightsp{\eta_u}  Y_u  \maprightsp{\pi_u}  \P^2,
\]
where $\pi_u$ is the double covering of $\P^2$ branched along $Q_u$, 
$\eta_u$ is  the double covering of $Y_u$ branched along $\pi_u\inv (Q_u)$, and 
 $\gamma_u=\pi_u\circ \eta_u$ is the cyclic covering of degree $4$   of $\P^2$ branched along $Q_u$.
We put
\[
\SY_u:=H^2(Y_u, \Z),
\]
which is a unimodular lattice of rank $8$
with the cup-product $\intfY{\phantom{i}, \phantom{i}}$.
Let $h_u\in \SY_u$ be  the class of the pull-back of a line on $\P^2$ by $\pi_u$.
It is well known that 
$Y_u$ is a del Pezzo surface of degree $2$
with the anti-canonical class $h_u$.
(See~\cite[Chapters~6 and~8]{DolgachevBook} about del Pezzo surfaces.)
On the other hand,  the surface $X_u$ is a $K3$ surface.
Let $\intfX{\phantom{i}, \phantom{i}}$ denote the cup product of $H^2(X_u, \Z)$, and 
let $\tilh_u$ be the class $\eta_u^*(h_u)$.
Then $\tilh_u$ is an ample class of degree $\intfX{\tilh_u, \tilh_u}=4$.
\par
It is classically known that 
every  smooth quartic curve $Q_u$ has exactly $28$ bitangents.
Moreover, if $u$ is general in $\UUU$,
all bitangents $l$ of $Q_u$  are ordinary, that is,
$l$ is tangent to $Q_u$ at distinct two points. 
\begin{definition}
A reduced conic $\barc\subset \P^2$ is called a \emph{splitting  conic} of $Q_u$ 
if the intersection multiplicity of $Q_u$ and $\barc$ is even at  each intersection point.
\end{definition}
A smooth conic $\barc$ is splitting
if and only if  $\pi_u^*(\barc)\subset Y_u$ has two irreducible components. 
A singular reduced conic   $\barc$ is splitting
if and only if $\barc$ is a union of two distinct  bitangents. 
\par
It is easy to see that a smooth conic $\barc=\{g=0\}$ 
defined by $g\in \Gamma(2)$ is a splitting conic of $Q_u=\{\varphi=0\}$
defined by $\varphi\in \Gamma(4)$
if and only if there exist polynomials $f\in \Gamma(2)$ and $q \in \Gamma(2)$ such that
$\varphi=fg+q^2$.
By an easy dimension counting,
we see the following:
\begin{lemma}\label{lem:2222or422}
Suppose that $u$ is general in $\UUU$.
Let  $\barc\subset \P^2$ be a smooth splitting conic of $Q_u$.
Then 
the intersection multiplicities of $Q_u$ and $\barc$ are either $(2, 2, 2, 2)$ or $(4, 2,2)$.
\qed
\end{lemma}
\begin{definition}
A smooth splitting conic $\barc\subset \P^2$ of $Q_u$ is called a \emph{$4$-tangent conic} (resp.~a \emph{$3$-tangent conic}) of $Q_u$ 
if the intersection multiplicities of $Q_u$ and $\barc$ are $(2,2,2,2)$ (resp.~$(4,2,2)$).
\end{definition}
\par
The following is easy to verify. 
The results are summarized in Table~\ref{table:pullbacks}. 
\begin{table}
{\small
\[
\renewcommand{\arraystretch}{1.4}
\begin{array}{l | l |l}
\P^2 & Y_u  &X_u \\
\hline 
\barl &\pi_u^*(\barl)= l+l\sprime  &\gamma_u^*(\barl)=\till+\till\sprime \\
&\intfY{l, l}=\intfY{l\sprime, l\sprime}=-1 &\intfX{\till, \till}=\intfX{\till\sprime, \till\sprime}=-2  \\
&\intfY{l, l\sprime}=2 &\intfX{\till, \till\sprime}=4 \\
&\intfY{h_u, l}=\intfY{h_u, l\sprime}=1 &\intfX{\tilh_u, \till}=\intfX{\tilh_u, \till\sprime}=2  \\
\hline
\barc &\pi_u^*(\barc)= c+c\sprime  &\gamma_u^*(\barc)= \tilc+\tilc\sprime \\
&\intfY{c, c}=\intfY{c\sprime, c\sprime}=0 &\intfX{\tilc, \tilc}=\intfX{\tilc\sprime, \tilc\sprime}=0  \\
&\intfY{c, c\sprime}=4 &\intfX{\tilc, \tilc\sprime}=8 \\
&\intfY{h_u, c}=\intfY{h_u, c\sprime}=2 &\intfX{\tilh_u, \tilc}=\intfX{\tilh_u, \tilc\sprime}=4
\end{array}
\]
}
\vskip 3pt
\caption{Pull-backs of bitangents and $4$-tangent conics}\label{table:pullbacks}
\end{table}
\begin{proposition}\label{prop:lifts}
{\rm (1)} 
Let $\barl$ be an ordinary  bitangent of $Q_u$.
Then $\pi_u^*(\barl)$ is a union of two smooth rational curves $l$ and $l\sprime$ on $Y_u$  with self-intersection $-1$ 
that intersect at two points transversely,
and $\gamma_u^*(\barl)$ is a union of two smooth rational curves $\till$ and $\till\sprime$ on $X_u$  with self-intersection $-2$ 
that intersect at two points with intersection multiplicities $(2,2)$.
\par
{\rm (2)} 
Let $\barc$ be a $4$-tangent conic of $Q_u$.
Then $\pi_u^*(\barc)$ is a union of two smooth rational curves $c$ and $c\sprime$ on $Y_u$  with self-intersection $0$ 
that intersect at four points transversely,
and $\gamma_u^*(\barc)$ is a union of two smooth elliptic curves $\tilc$ and $\tilc\sprime$ on $X_u$  with self-intersection $0$ 
that intersect at four points with intersection multiplicities $(2,2,2,2)$.
\qed
\end{proposition}
\begin{definition}\label{def:lifts}
A curve $l$ on $Y_u$ is called a \emph{$Y$-lift} of 
a bitangent $\barl$ of $Q_u$ if 
$\pi_u$ maps $l$ to $\barl$ isomorphically.
We also say that 
a curve $c$ on $Y_u$ is  a \emph{$Y$-lift} of 
a splitting conic $\barc$ of $Q_u$ if 
$\pi_u$ maps $c$ to $\barc$ isomorphically.
\end{definition}
\section{Monodromy}\label{sec:monodromy}
Let $u$ be a point of $\UUU$.
It is well known 
that the lattice 
$\SY_u=H^2(Y_u, \Z)$ is isomorphic to the lattice of rank $8$
whose Gram matrix is  the diagonal matrix  
$\diag (1, -1, \dots, -1)$,
and that 
the orthogonal complement 
\[
\Sigma_u := (\Z h_u \inj \SY_u)\sperp
\]
of the ample class $h_u$ in $\SY_u$ 
is isomorphic to the negative-definite root lattice of type $E_7$.
The deck transformation 
\[
\iota_u\colon Y_u\to Y_u
\]
of  $\pi_u\colon Y_u\to\P^2$
acts on $\Sigma_u$ as $-1$.
Note that the group $\OG(\Sigma_u)$  of isometries of  $\Sigma_u$
is equal to the Weyl group
$W(E_7)$, which is of order $2903040$. 
Hence there exists an injective  homomorphism  
\begin{equation}\label{eq:restoW}
\OG(\SY_u, h_u):=\set{g\in \OG(\SY_u)}{h_u^g=h_u} \;\;\inj \;\; W(E_7).
\end{equation}
%
%
It is easy to check that 
the action on $\Sigma_u$ of 
each of the standard generators of $W(E_7) $ lifts 
to an isometry of $\SY_u$ that fixes $h_u$.
Hence the homomorphism~\eqref{eq:restoW} is in fact  an isomorphism.
The family of lattices $\set{\SY_u}{u\in \UUU}$ 
forms a locally constant system  
\[
\SSSY\to \UUU.
\]
Let $b$ be a general point of $\UUU$,
which will serve as a base point of  $\UUU$.
The  monodromy action
of $\pi_1(\UUU, b)$ on the lattice $\SY_b$
preserves $h_b\in \SY_b$.
%
\begin{theorem}[Harris~\cite{Harris1979}]\label{thm:monodromy}
The monodromy homomorphism
\begin{equation}\label{eq:mon}
\pi_1(\UUU, b) \to \OG(\SY_b, h_b)\cong W(E_7)
\end{equation}
is surjective.
\end{theorem}
The original statement in~\cite{Harris1979}
is not on the monodromy action on the lattice $\SY_b$,
but on the 
Galois group $W(E_7)/\{\pm 1\}\cong \GO_6(\F_2)$ of bitangents of $Q_b$.
Moreover the proof in~\cite{Harris1979}  
is via the proof of a similar result
on cubic surfaces with $E_7$ replaced by $E_6$.
Hence we give a direct and simple proof of Theorem~\ref{thm:monodromy} below.
\par
For the proof, we prepare some more notation,
which will be used throughout this paper. 
We denote by 
$\barL_u$  the set of bitangents of  $Q_u$, and 
$L_u$  the set of $Y$-lifts of  bitangents of $Q_u$.
Let $\Sigma_u\dual$ denote  the dual lattice of $\Sigma_u$.
By identifying  $l\in L_u$ with  its class $[l]\in \SY_u$,
we have an identification 
\begin{eqnarray}
L_u &\;=\;& \set{v\in \SY_u}{\intfY{v, h_u}=1, \;\intfY{v, v}=-1} \label{eq:Lu} \\
&\;\cong\;& \set{v\in \Sigma_u\dual}{\intfY{v, v}=-3/2}, \nonumber
\end{eqnarray}
where the second bijection is obtained by  the 
orthogonal projection
$\SY_u\to \Sigma_u\dual$.
We put $\barSY_u:=\SY_u/\gens{\iota_u}$, and consider the commutative diagram
\begin{eqnarray}\label{eq:diagLLSSYY}
\begin{array}{ccc}
L_u & \inj & \SY_u \\
\mapdown  &  &\mapdown \\ 
\barL_u &\inj &\barSY_u,
\end{array}
\end{eqnarray}
where vertical arrows are quotient maps by the involution $\iota_u$.
Since the  action of $\pione(\UUU, b)$
on $\SY_b$ commutes with $\iota_b$,
we have a monodromy  action of $\pione(\UUU, b)$ 
on $ \barSY_b$.
Thus we obtain a diagram
\begin{eqnarray}\label{eq:caldiagLLSSYY}
\begin{array}{ccc}
\LLL & \inj & \SSSY \\
\mapdown   &  &\mapdown \\ 
\barLLL &\inj& \barSSSY
\end{array}
\end{eqnarray}
of locally constant systems over $\UUU$  parameterizing the diagram~\eqref{eq:diagLLSSYY} over $\UUU$,
where vertical arrows are quotient maps by the family of involutions 
\begin{equation*}\label{eq:iotaUUU}
\iotaUUU:=\set{\iota_u}{u\in \UUU}.
\end{equation*}
Note that $\barLLL$ is the space parameterizing all
bitangents of smooth quartic curves.
\begin{proof}[Proof of Theorem~\ref{thm:monodromy}]
Let $\disjsevenset_u$ (resp.~$\disjsevenlist_u$) be the set 
of non-ordered $7$-tuples $\{l_1, \dots, l_7\}$ 
(resp.~ordered $7$-tuples $[l_1, \dots, l_7]$)
of elements $l_1, \dots, l_7\in L_u$ such that $\intfY{l_i, l_j}=0$ for $i\ne j$. 
By~\eqref{eq:Lu}, we can enumerate all  elements  of $\disjsevenset_u$.
It turns out that  
$|\disjsevenset_u|=576$, 
and hence 
\begin{equation}\label{eq:disjsize}
|\disjsevenlist_u|=576\cdot 7!=2903040 =|W(E_7)|.
\end{equation}
(See also Remark~\ref{rem:576}.)
For $7$-tuples $\lambda=[l_1, \dots, l_7]$ and 
$\lambda\sprime=[l\sprime_1, \dots, l\sprime_7]$ in $\disjsevenlist_u$,
there exists a unique isometry
$g_{\lambda, \lambda\sprime}\in \OG(\SY_u\tensor \Q)$
such that
\[
g_{\lambda, \lambda\sprime}(h_u)=h_u, \quad
g_{\lambda, \lambda\sprime}(l_i)=l\sprime_i \quad(i=1, \dots, 7).
\]
It is enough show that, 
when $u=b$, 
these elements 
$g_{\lambda, \lambda\sprime}$ are 
contained in the image of the monodromy~\eqref{eq:mon}.
Indeed, by~\eqref{eq:disjsize}, 
this claim implies that 
these isometries  $g_{\lambda, \lambda\sprime}$ 
constitute  the whole group $\OG(\SY_b, h_b)\cong W(E_7)$.
To prove this claim,
it is enough to show  that 
$\pione (\UUU, b)$ acts on $\disjsevenlist_b$ transitively by the monodromy, 
or equivalently, to show  that   the total space $\disjsevenlistfam$ of 
the locally constant system 
\[
\disjsevenlistfam \to \UUU
\]
obtained from the family $\set{\disjsevenlist_u}{u\in \UUU}$ is connected.
\par
Let $\lambda=[l_1, \dots, l_7] $ be a point of $\disjsevenlistfam$
over   $u\in \UUU$. 
Contracting the  $(-1)$-curves $l_1, \dots, l_7$, we obtain a birational morphism
\[
\bl_{\lambda}\colon Y_u \to \PP_{\lambda}
\]
to a projective plane $\PP_{\lambda}$.
We put  $\beta_{\lambda}:=[\bl_{\lambda}(l_1), \dots, \bl_{\lambda}(l_7) ]$.
Conversely, we fix a projective plane  $\PP$,
and let $\PPseven$ denote the set of 
ordered $7$-tuples $[p_1, \dots, p_7]$ of distinct points
of $\PP$.
For a general point $\beta=[p_1, \dots, p_7]$ of $ \PPseven$,
let 
\[
\bl^{\beta}\colon Y^{\beta}\to \PP
\]
be the blowing-up
at the points $p_1, \dots, p_7$.
Then $Y^{\beta}$ is a del Pezzo surface of degree $2$,
and the complete linear system 
of the anti-canonical divisor on $Y^{\beta}$ gives 
a double covering $Y^{\beta}\to \P^2$ branched along a smooth quartic curve $Q^{\beta}$
such that each of the $7$ exceptional curves 
over  $p_1, \dots, p_7$
is a $Y$-lift of a bitangent of $Q^{\beta}$.
Hence there exist a point $\lambda\in  \disjsevenlistfam$
and an isomorphism $\PP_{\lambda} \cong \PP$
that maps $\beta_{\lambda}$ to $\beta$.
\par
We put
\[
\III:=\bigset{(\lambda, \gamma, \beta)}{\;\;\parbox{7cm}{$\lambda \in \disjsevenlistfam$, $\beta\in \PPseven$, and $\gamma$ is 
an isomorphism $\PP_{\lambda} \cong \PP$
that maps $\beta_{\lambda}$ to $\beta$}}.
\]
Then the projection $\III\to \disjsevenlistfam$ is surjective with fibers isomorphic to $\PGL_3(\C)$,
whereas  the projection $\III\to \PPseven$ is dominant with fibers isomorphic to $\PGL_3(\C)$.
Since $\PPseven$ and $\PGL_3(\C)$ are connected, 
we see that $\disjsevenlistfam$  is connected.
\end{proof}
\section{Family of $4$-tangent conics}\label{sec:familyofconics}
In this section,
we construct a space $\barCCC$
parameterizing all $4$-tangent conics of smooth quartic curves.
\par
Let $\barC_u$  denote the set of $4$-tangent conics of $Q_u$, 
and let $C_u$ be the set of $Y$-lifts of $4$-tangent conics of $Q_u$.
We put
\begin{equation}\label{eq:Fu}
F_u:=\set{v\in \SY_u}{\intfY{v, h_u}=2, \intfY{v, v}=0} 
\;\cong\;  \set{v\in \Sigma_u}{\intfY{v, v}=-2}, 
\end{equation}
where the second bijection is given by  the orthogonal projection $\SY_u\to \Sigma_u\dual$.
As was shown in Table~\ref{table:pullbacks}, we have $[c]\in F_u$ for any $c\in C_u$.
Note that  $|F_u|=126$,
the number of roots of the root lattice $\Sigma_u$ of type $E_7$.
We  put
\[
\barF_u:=F_u/\gens{\iota_u} \;\; \subset \;\; \barSY_u.
\]
Then we have a commutative diagram
\begin{equation}\label{eq:CCFFdiag}
\begin{array}{ccccc}
C_u&\maprightsp{ \Phi_u} & F_u & \inj  & \SY_u \\
\mapdown  & & \mapdown  & & \mapdown  \\
\barC_u&\maprightsp{ \barPhi_u} & \barF_u & \inj  & \barSY_u
\end{array}
\end{equation}
where $ \Phi_u\colon C_u\to F_u$ is given by $c\mapsto [c]\in \SY_u$,
and the vertical arrows are quotient by the involution $\iota_u\colon Y_u\to Y_u$.
We have locally constant systems 
$\FFF\to \UUU$  and  $\barFFF\to \UUU$ 
obtained from  the families  $\set{F_u}{u\in \UUU}$
and $\set{\barF_u}{u\in \UUU}$.
\begin{theorem}\label{thm:barCCC}
There exists a commutative diagram 
\begin{equation}\label{eq:calCCFFdiag}
\begin{array}{ccccc}
\CCC&\maprightsp{ \PhiUUU} & \FFF & \inj  & \SSSY \\
\mapdown  & & \mapdown  & & \mapdown  \\
\barCCC&\maprightsp{ \barPhiUUU} & \barFFF & \inj  & \barSSSY
\end{array}
\end{equation}
of morphisms over $\UUU$ 
that parameterizes the diagrams~\eqref{eq:CCFFdiag} over $\UUU$.
The morphisms   $\PhiUUU\colon \CCC\to \FFF$ and $\barPhiUUU\colon \barCCC\to \barFFF$  are smooth and surjective,
and every fiber of them is a Zariski open subset of $\P^1$.
\end{theorem}
For the proof, we use the double covering $\eta_u\colon X_u\to Y_u$ of $Y_u$ 
by the  $K3$ surface $X_u$.
We consider the N\'eron-Severi lattice 
\[
\SX_u:=H^2(X_u, \Z)\cap H^{1,1}(X_u)
\]
with the intersection form $\intfX{\,,\,}$.
Then $\eta_u$ induces 
a embedding of the lattice
\[
\eta_u^*\colon \SY_u(2)\inj \SX_u,
\]
where $\SY_u(2)$ is the lattice obtained from $\SY_u$ by multiplying the intersection form by  $2$.
Let $j_u\colon  X_u\to X_u$ be a generator of 
the cyclic group $\Gal(X_u/\P^2)$ of order $4$.
Then $\eta_u\colon X_u\to Y_u$
is the quotient morphism by $j_u^2$.
Hence $j_u^2$ acts on the image of $\eta_u^*\colon \SY_u(2)\inj \SX_u$ trivially.
\begin{proof}[Proof of Theorem~\ref{thm:barCCC}]
Note that the family of involutions $\iotaUUU=\set{\iota_u}{u\in \UUU}$  
acts on $\FFF$ over $\UUU$ 
without fixed points.
Hence, 
if the parameterizing space  $\PhiUUU\colon \CCC\to \FFF$ 
of $\Phi_u\colon C_u\to F_u$ is constructed,
then $\barPhiUUU\colon \barCCC\to\barFFF$ is constructed  as a quotient of 
$\PhiUUU\colon \CCC\to \FFF$ by $\iotaUUU$.
\par
Let $u$ be an arbitrary point of $\UUU$,
and let $v$ be an element of $F_u\subset \SY_u$.
We put  $\tilv:=\eta_u^*(v)\in \SX_u$.
We can easily confirm that
there exist exactly $6$ pairs  
$\{l_i, l\sprime_i\}$ ($i=1, \dots, 6$) of $Y$-lifts of bitangents of $Q_u$
such that $\intfY{l_i, l_i\sprime}=1$ and $v=[l_i]+[l_i\sprime]$,
and that these $12$  curves $l_1, l_1\sprime, \dots, l_6, l_6\sprime$
are distinct. 
Hence the complete linear system on the $K3$ surface $X_u$ 
corresponding to $\tilv\in \SX_u$ has no fixed components.
The class $\tilv$ is primitive in $\SX_u$
with  $\intfX{\tilv, \tilv}=0$ and $\intfX{\tilh_u, \tilv}=4$.
Therefore there exists an elliptic fibration on $X_u$
such that the class of a fiber is equal to $\tilv$.
We denote this elliptic fibration by $\phi_v\colon X_u\to \P^1$ .
\par
If $c\in C_u$,
then $\eta_u^*(c)$ is an elliptic curve,
and hence $\eta_u^*(c)$ is a smooth fiber of an elliptic fibration of $\phi_v\colon X_u\to \P^1$, 
where $v=[c]$.
Conversely, suppose that $(u, v)\in \FFF$, and 
let  $f$ be a smooth fiber of the  elliptic fibration $\phi_v\colon X_u\to \P^1$.
We denote by  $a\subset \P^2$  the plane curve $\gamma_u(f)$
with the reduced structure.
Let $d$ be the degree of $a$,
and $\delta$ the mapping degree of $\gamma_u|f\colon f\to a$.
Since $\intfX{\tilh_u, f}=4$ and $\gamma_u\colon X_u\to \P^2$ is Galois,
we have $(d, \delta)=(1,4)$, $(2,2)$, or $(4,1)$.
If $(d, \delta)=(1,4)$, 
then $f=\gamma_u\inv(a)$ is invariant under the action of $\Gal(X_u/\P^2)=\gens{j_u}$,
and hence the class $[f]\in \SX_u$ is a non-zero multiple of $\tilh_u$,
which contradicts $\intfX{f, f}=0$.
If $(d, \delta)=(4,1)$, 
then $f$, $j_u(f)$,  $j_u^2(f)$, $j_u^3(f)$
are distinct curves that intersect over the points  of $a\cap Q_u$.
On the other hand, since $[f]=\eta_u^* (v)\in \Image \eta_u^*$,
we have $j_u^{*2}([f])=[f]$.
This contradicts $\intfX{f, f}=0$.
Hence $(d, \delta)=(2,2)$, and 
we see that $a$ is a smooth splitting conic.
Note that $a$ is $4$-tangent,
because otherwise $f$ would be singular.
Thus we have proved that $c\mapsto \eta_u^*(c)$
gives a bijection from $C_u$ to the union of the sets of smooth fibers
of elliptic fibrations $\phi_v\colon X_u\to \P^1$, where  $v$ runs through $F_u$.
\par
Let $\XXX\to \UUU$ be the universal family of $\set{X_u}{u\in \UUU}$,
and let $\pi_{\FFF}\colon \FFF\times_{\UUU}\XXX\to \FFF$ be the pull-back of $\XXX\to \UUU$ by $\FFF\to\UUU$.
Let $\MMM$ be a line bundle on $\FFF\times_{\UUU}\XXX$
such that
the class $[\MMM|X_u]\in \SX_u$ of the line bundle 
$\MMM|X_u$ on  $\pi_{\FFF}\inv (u,v)= X_u$
is equal to $v\in F_u$.
Then $\pi_{\FFF*} \MMM\to \FFF$ is a vector bundle of rank $2$.
The fiber over $(u, v)\in \FFF$ 
of the $\P^1$-bundle $\P_*( \pi_{\FFF*} \MMM)\to \FFF$  
is identified with the base curve of the elliptic fibration
$\phi_v\colon X_u\to \P^1$.
We can construct $\CCC$ as the open subset of $\P_*( \pi_{\FFF*} \MMM)$
consisting of non-critical points of $\phi_v\colon X_u\to \P^1$.
\end{proof}
The non-singular varieties $\CCC$ and $\barCCC$
parameterize all pairs $(u, c)$ and  $(u, \barc)$, respectively, 
where $u\in \UUU$ and $c \in C_u$, $\barc\in \barC_u$.
Since $\Phi_u$ and $\barPhi_u$ have connected fibers, 
we can regard $F_u$  as the set of connected families of $Y$-lifts of 
$4$-tangent conics of $Q_u$,
and  $\barF_u$ as the set of connected families of
$4$-tangent conics.
The following observation obtained in the proof of Theorem~\ref{thm:barCCC} 
will be  used in the next section.
\begin{proposition}\label{prop:nobasepts}
Every connected family $[c]\in F_u$ of $Y$-lifts of  splitting conics is a pencil with no base points.
\qed
\end{proposition}
\begin{remark}\label{rem:Pic}
A line section $\Lambda_u$ of $Q_u\subset \P^2$ is a canonical class of the genus $3$ curve $Q_u$.
Let $\Pic^0(Q_u)$ be the Picard group of line bundles of degree $0$  of  $Q_u$,
and let $\Pic^0(Q_u)[2]$ be the  subgroup of $2$-torsion points of $\Pic^0(Q_u)$.
For a $4$-tangent conic $\barc$ of $Q_u$,
let $\Theta_u(\barc)$ be 
the reduced part of the divisor $\barc\cap Q_u$ of $Q_u$.
Then the class of the divisor $\Theta_u(\barc)-\Lambda_u$ of degree $0$ is a point of $\Pic^0(Q_u)[2]-\{0\}$,
and this correspondence gives a bijection $\barF_u\cong \Pic^0(Q_u)[2]-\{0\}$.
\end{remark}
\section{Proof of the main results}\label{sec:ZZZ}
In this section, 
we construct the space $\ZZZ\spmn$ parameterizing all  $\Qmn$-curves,
and prove Theorems~\ref{thm:main}~and~\ref{thm:main2}.
\subsection{Deformation types}\label{subsec:Deftypes}
We fix some notation.
For a  set $A$, let $S^k(A)$ denote  the symmetric product 
$A^k/\SSSS_k$, 
where $A^k=A\times \dots\times A$ ($k$ times), 
and let $S_0^k(A)$ denote the complement in $S^k(A)$ of the image of the big diagonal in $A^k$. 
\par
For a morphism  $\AAA\to \UUU$,
let $\SSS^k(\AAA)$ denote  the symmetric product
$\AAA^k/\SSSS_k$, 
where $\AAA^k:=\AAA\times_{\UUU} \dots\times_{\UUU} \AAA$  ($k$ times), 
and  let $\SSS_0^k(\AAA)$ denote the complement in $\SSS^k(\AAA)$ of the image of the big diagonal
in  $\AAA^k$. 
Note that,  if $\AAA$ is smooth over $\UUU$ with relative dimension $1$,
then $\SSS^k(\AAA)$ is smooth over $\UUU$ with relative dimension $k$.
\par
Recall that   $\barLLL\to \UUU$ and $\barCCC\to \UUU$  are 
the spaces parameterizing
all  bitangents and 
all $4$-tangent conics of smooth quartic curves,
respectively.
We put
\[
\ZZZ\sp{\prime (m, n)}:=\SSS_0^m(\barLLL)\times_{\UUU} \SSS_0^n(\barCCC),
\]
which is the space parameterizing  all curves $Z\sprime$,
where $Z\sprime$ is a union    
of a smooth quartic curve $Q$, $m$ distinct bitangents of $Q$, and $n$ distinct $4$-tangent conics of $Q$.
Now we can construct  the parameter space 
\[
\varpi\colon \ZZZ\spmn\to \UUU
\]
of $\Qmn$-curves
as the open subvariety of $\ZZZ\sp{\prime (m, n)}$ 
consisting of points
corresponding to plane curves 
$Z\sprime$ satisfying conditions  (i), (ii), (iii) in Definition~\ref{def:Qmn}.
For a point $\zeta\in \ZZZ\spmn$,
we denote by $Z_{\zeta}$ the $\Qmn$-curve  corresponding to $\zeta$. 
\par
For $u\in \UUU$,
we put
\[
P\spmn_u:=S_0^m(\barL_u)\times S^n(\barF_u)\;\;\subset\;\; S_0^m(\barSY_u)\times S^n(\barSY_u).
\]
The size of $P\spmn_u$ is equal to $d\spmn$ defined  by~\eqref{eq:dmn}.
Then we obtain  a finite \'etale covering 
\[
\rho\colon \PPP\spmn:=\SSS_0^m(\barLLL)\times_{\UUU} \SSS^n(\barFFF) \to \UUU
\]
of degree $d\spmn$ parameterizing 
the family $\set{P\spmn_u}{u\in \UUU}$.
Using  
$\barPhiUUU\colon \barCCC\to \barFFF$
in Theorem~\ref{thm:barCCC},
we have a morphism
$\theta\sprime\colon \ZZZ\sp{\prime (m, n)}\to \PPP\spmn$.
Restricting $\theta\sprime$ to the open subvariety $\ZZZ\spmn \subset  \ZZZ\sp{\prime (m, n)}$, 
we obtain a morphism
\[
\theta\colon \ZZZ\spmn\to \PPP\spmn,
\]
which maps $\zeta\in \ZZZ\spmn$ to 
\begin{equation}\label{eq:rhoZ}
\theta(\zeta):=
(\{\barl_1, \dots, \barl_m\}, [[\barc_1], \dots, [\barc_n]]) \;\in\;P\spmn_{\varpi(\zeta)},
\end{equation}
where $Z_{\zeta}$ has the irreducible components  as in~\eqref{eq:Z}.
Thus we obtain the following  commutative diagram.
\[
\begin{array}{ccccc}
\ZZZ\spmn && \raise 2pt \hbox{$\maprightsp{\theta} $}&& \PPP\spmn \\
&\hskip -10pt \llap{\tiny $\varpi$}\searrow && \swarrow\rlap{\tiny $\rho$} \hskip -10pt \\
&& \UUU &&
\end{array}
\]
To investigate the image of $\theta$, 
we put
\[
\UUU\sprime
:=\sethd{u\in \UUU}{8cm}{every bitangent of $Q_u$ is ordinary, and their union has only ordinary nodes as its singularities},
\]
which is a Zariski open dense subset of  $\UUU$.
\begin{lemma}\label{lem:smooth}\label{lem:UUUsprime}
The  morphism $\theta\colon \ZZZ\spmn \to \PPP\spmn$ is 
smooth with each non-empty  fiber being an irreducible variety of 
dimension $n$.
The image of $\theta$ contains $\rho\inv (\UUU\sprime)$.
In particular, the morphism $\theta$ is dominant.
\end{lemma}
\begin{proof}
Since $\barPhiUUU\colon \barCCC\to \barFFF$ is smooth and surjective with each  fiber being
a Zariski open subset of $\P^1$,
the morphism $\theta\sprime\colon \ZZZ\sp{\prime (m, n)}\to \PPP\spmn$ 
is smooth and surjective 
with each fiber being an irreducible variety of 
dimension $n$. 
Suppose that $u\in \UUU\sprime$,
and let $p:= (\{\barl_1, \dots, \barl_m\}, [[\barc_1], \dots, [\barc_n]])$
be a point of $P\spmn_u$.
By  Proposition~\ref{prop:nobasepts} and Bertini's theorem,
if we choose each $4$-tangent conic $\barc_j\sprime$ in the connected family $[\barc_j]\in \barF_u$ generally,
the curve $Q_{u}+\sum \barl_i+ \sum \barc_j\sprime$
satisfies conditions (ii) and (iii) in Definition~\ref{def:Qmn}.
Hence $\theta\inv(p)=\theta^{\prime-1}(p)\cap \ZZZ\spmn$ is non-empty.
\end{proof}
\begin{proof}[Proof of Theorem~\ref{thm:main2}]
By Lemma~\ref{lem:smooth}, 
the connected components of $\ZZZ\spmn$
are  in bijective correspondence with 
the connected components of $\PPP\spmn$,
and hence with the $\pi_1(\UUU, b)$-orbits
in $P\spmn_b$.
By Theorem~\ref{thm:monodromy},
the number $N\spmn$ of $\pi_1(\UUU, b)$-orbits in $P\spmn_b$
satisfies~\eqref{eq:estimate}, because $|W(E_7)/\{\pm 1\}|=1451520$.
\end{proof}
\subsection{Computation of $N\spmn$}\label{subsec:computationalexample}
Recall that $\Sigma_b$ is a negative-definite root lattice of type $E_7$.
Let $\Sigma$ be the negative-definite root lattice of type $E_7$ with the standard basis, and let $\Sigma\dual$  be its dual.
According to~\eqref{eq:Lu} and~\eqref{eq:Fu},
we define the subsets 
\[
\barL:= \set{v\in \Sigma\dual}{\intf{v, v}=-3/2}/\{\pm 1\}
\]
 of $\barSigma\dual:=\Sigma\dual /\{\pm 1\}$,
  and 
 \[
\barF:= \set{v\in \Sigma}{\intf{v, v}=-2}/\{\pm 1\}
\]
 of $\barSigma:=\Sigma/\{\pm 1\}$.
We then put 
\[
P\spmn:= S_0^m(\barL)\times S^n(\barF).
\]
The group $W(E_7)$ is generated by seven standard reflections.
The permutations  on  $\barL$ and on $\barF$   induced by these generators are easily calculated.
Hence the permutations on $P\spmn$  induced by these generators are also calculated.
Thus we can compute the orbit decomposition of $P\spmn$ by $W(E_7)$,
and obtain the number $N\spmn$ of deformation types of $\Qmn$-curves.
\begin{example}
The size $d\sp{(4,0)}$ of  $P^{(4,0)}$ is $20475$.
The group $W(E_7)$ decomposes this set into three orbits of sizes $315, 5040, 15120$.
Hence $N^{(4,0)}=3$.
\end{example}
\begin{example}
The size $d\sp{(0,4)}$ of  $P^{(0,4)}$ is $720720$.
The group $W(E_7)$ decomposes this set into $30$ orbits as follows:
{\small
\begin{eqnarray*}
720720= 63+ 
 945\times 3+
 1008\times 2+
 1890+2016+3780\times 2+ 5040\times 2+ 10080+\\+11340+15120\times 5+
 22680+30240\times 5+ 45360\times 2+90720+
      120960\times 2.
\end{eqnarray*}}%
Hence $N^{(0,4)}=30$.
\end{example}
\begin{example}
The size $d\sp{(2,2)}$ of  $P^{(2,2)}$ is $762048$.
The group $W(E_7)$ decomposes this set into 
$23$ orbits  as follows:
{\small
\begin{eqnarray*}
762048=378+1890+3780\times 3+6048+7560\times2 + 12096\times2 + 15120+ 22680+\\+ 30240\times 3 +45360\times 2+ 60480\times 4 +
120960\times 2.
\end{eqnarray*}}%
Hence $N^{(2,2)}=23$.
\end{example}

\begin{remark}
For the computation, we used~{\tt GAP}~\cite{GAP},
which is  good at computations of permutation groups.
\end{remark}
\subsection{Real quartic curves}\label{subsec:real}
\begin{definition}
Note that $H^2(\P^2, \Z)\cong \Z$ has a canonical generator, that is, 
the class of a line.
Let $C$ and $C\sprime$ be plane curves with the same homeomorphism type.
A homeomorphism $\sigma\colon (\P^2, C)\isom (\P^2, C\sprime)$ is said to be
\emph{orientation-preserving} (resp.~\emph{orientation-reversing}) 
if the action of $\sigma$ on $H^2(\P^2, \Z)$  is the identity (resp.~the multiplication by  $-1$).
\end{definition}
\begin{example}\label{example:alphastar}
Suppose that $\Qmn$-curves $Z_{\zeta}$ and $Z_{\zeta\sprime}$ are of the same deformation type.
Let $\alpha\colon I\to \ZZZ\spmn$ be a path from $\zeta$ to $\zeta\sprime$,
where $I:=[0,1]\subset\R$.
By the parallel transport along $\alpha$,
we obtain a homeomorphism $\alpha_{*}\colon (\P^2, Z_{\zeta})\isom (\P^2, Z_{\zeta\sprime})$.
It is obvious that $\alpha_{*}$ is orientation-preserving.
\end{example}
\begin{proposition}\label{prop:selfhomeo}
Every $\Qmn$-curve $Z_{\zeta}$ admits  an orientation-reversing self-homeomorphism 
$(\P^2, Z_{\zeta}) \isom (\P^2, Z_{\zeta})$.
\end{proposition}
For the proof of Proposition~\ref{prop:selfhomeo},
we use  a classical result on real quartic curves.
We give a structure of the $\R$-scheme to $\P^2$.
We denote by $\Gamma_{\R}(d)$ the space of homogeneous polynomials of degree $d$ on $\P^2$
with real coefficients,
and consider the real projective space $\P_{*} (\Gamma_{\R}(4))$ as a closed subset of 
$\P_{*} (\Gamma(4))$.
We then put
\[
\UUUR:=\UUU\cap \P_{*} (\Gamma_{\R}(4)).
\]
The topological types of smooth real quartic curves are classified by Zeuthen and Klein, 
and the result is summarized  in~\cite[Theorem 1.7]{PSV2011}.
Using this result, we obtain the following:
\begin{theorem}[Zeuthen~(1873) and Klein~(1876)]\label{thm:ZK}
There exists a unique  connected component $\UUURf$ of $\UUUR$ 
consisting of points $u\in \UUUR$
such that the real plane curve $Q_{u}(\R)$ is a union  of $4$ ovals.
If $u\in \UUURf$,
the ovals in $Q_{u}(\R)$ are pairwise non-nested,
and every bitangent of $Q_u$ is defined over $\R$.
\qed
\end{theorem}
\begin{remark}
For  beautiful pictures of  real plane quartic curves with real $28$ bitangents,
see~\cite{Kuwata2005} and~\cite[Section 10.5]{MSBook2019}.
These pictures are in fact defined over $\Q$,
and were obtained by the theory of Mordell-Weil lattices.
\end{remark}
For an algebraic variety $V$ defined over $\R$,
we denote by $H^*(V, \Z)$ the cohomology ring
of the topological space $V(\C)$ of $\C$-valued points of $V$,
by 
\[
\tau_{V}\colon V(\C)\isom V(\C)
\]
the self-homeomorphism of $V(\C)$ 
obtained by the complex conjugation, 
and by $V_{\C}$ the variety $V\tensor_{\R}\C$ defined over $\C$.
Let $S$ be an algebraic surface defined over $\R$,
and let $C$ be a reduced irreducible  curve on $S_{\C}$.
Then there exists a unique  reduced irreducible curve $C\sprime$ on $S_{\C}$
such that $\tau_S$ induces an orientation-reversing 
homeomorphism $C(\C)\isom C\sprime(\C)$.
We denote this curve $C\sprime$ by  $\tau_S[C]$.
Then we have 
\[
[\tau_S[C]]=-\tau_S^*([C])
\]
 in $H^2(S, \Z)$.
If $C$ is also defined over $\R$, then $\tau_S[C]=C$,
and hence $\tau_S^*([C])=-[C]$.
If $H^2(S, \Z)$ is generated by classes 
of curves defined over $\R$,
then $\tau_S$ acts on $H^2(S, \Z)$ as the multiplication by $-1$,
and hence, for any curve $C$ (not necessarily defined over $\R$),
we have $[\tau_S[C]]=[C]$ in $H^2(S, \Z)$.
\begin{lemma}\label{lem:tausamefam}
Let $r$ be a point of $\UUURf$.
If $\barc$ is a $4$-tangent conic of $Q_r$,
then the $4$-tangent conic $\tau_{\P^2}[\barc]$ of $\tau_{\P^2}[Q_r]=Q_r$
is in the same connected family as $\barc$.
\end{lemma}
\begin{proof}
Note that, for $\varphi\in \Gamma_{\R}(4)$ and $x\in \P^2(\R)$,
the sign of $\varphi(x)$ is well-defined, 
because $\lambda^4>0$ for any $\lambda\in \R\sptimes$.
We choose a defining equation $\varphi\in \Gamma_{\R}(4)$ of $Q_r$ in 
such a way that $\varphi(x)>0$ for any point $x$ of $\P^2(\R)$ in the outside of the ovals of $Q_r(\R)$.
We let $Y_r$ be defined over $\R$ by $w^2=\varphi$,
and consider the self-homeomorphism  
$\tau_{Y}\colon Y_r(\C)\isom Y_r(\C)$
given by the complex conjugation.
For any bitangent $\barl$ of $Q_r$,
each of its $Y$-lifts $l$
satisfies $\tau_{Y}[l]=l$,
because $\varphi(x)\ge 0$ for any point $x$ of $\barl(\R)$.
Since the classes of these curves $l$ span $\SY_r=H^2(Y_r, \Z)$,
we see that $\tau_Y$ acts on $\SY_r$ as the multiplication by $-1$.
Therefore, for any curve $C$ on $Y_r$, we have $[\tau_Y[C]]=[C]$.
In particular, if $c\subset Y_r$ is a $Y$-lift of $\barc$,
then $\tau_Y[c]$ 
is a $Y$-lift of the $4$-tangent conic  $\tau_{\P^2}[\barc]$.
Then   $[\tau_Y[c]]=[c]$ in $\SY_r$ 
implies that $\tau_{\P^2}[\barc]$ and $\barc$ belong to 
the same connected family of 
$4$-tangent conics.
\end{proof}
\begin{proof}[Proof of Proposition~\ref{prop:selfhomeo}]
Since $\UUURf$ is  open in $\P_{*} (\Gamma_{\R}(4))$, 
it follows that $\UUURf$ is Zariski dense in $\UUU$, and hence 
there exists a point $r\in \UUURf \cap \UUU\sprime$.
By Lemma~\ref{lem:UUUsprime},
we see that 
$\varpi\inv (r)$ intersects every connected component of $\ZZZ\spmn$.
Let $\xi$ be a point of $\varpi\inv (r)$ that belongs to the same connected component 
as $\zeta$,
and let
\[
Z_{\xi}=Q_r+\barl_1\sprime+ \cdots +\barl_m\sprime+\barc_1\sprime+ \cdots +\barc_n\sprime
\]
be the decomposition of $Z_{\xi}$ into irreducible components.
Remark that $Q_r$ and all of its bitangents are defined over $\R$ by the definition of $\UUURf$ (see Theorem~\ref{thm:ZK}).
We choose a path $\alpha\colon I\to \ZZZ\spmn$
from $\zeta$ to $\xi$.
Then we obtain an orientation-preserving homeomorphism
\[
\alpha_{*}\colon (\P^2, Z_{\zeta})\isom (\P^2, Z_{\xi}).
\]
For simplicity,
we write $\tau$ instead of $\tau_{\P^2}$.
We have an orientation-reversing homeomorphism
\[
\tau \colon (\P^2, Z_{\xi})\isom (\P^2, \tau[Z_{\xi}])
\]
obtained by the complex conjugation.
Since
\[
\tau[Z_{\xi}]=
Q_r+\barl_1\sprime+ \cdots +\barl_m\sprime+\tau[\barc_1\sprime]+ \cdots +\tau[\barc_n\sprime],
\]
and,
for $j=1, \dots, n$,
the $4$-tangent conic  $\tau[\barc_j\sprime]$  
of $\tau[Q_r]=Q_r$ belongs to the same connected family as $\barc_j\sprime$ by Lemma~\ref{lem:tausamefam},
we see that $\Qmn$-curves $\tau[Z_{\xi}]$ and $Z_{\xi}$ have the same  deformation type,
and we have 
an orientation-preserving homeomorphism
\[
\beta_{*}\colon (\P^2, \tau[Z_{\xi}])\isom (\P^2, Z_{\xi}).
\]
Composing $\alpha_{*}$,  $\tau$, $\beta_{*}$ and $\alpha_{*}\inv$,
we obtain an orientation-reversing self-homeomorphism 
of $(\P^2, Z_{\zeta})$.
\end{proof}
%

%
\subsection{Homeomorphism types}\label{subsec:Homeotypes}
Let $\zeta$ be a point of $\ZZZ\spmn$ 
such that $Z_{\zeta}$ has the decomposition 
\begin{equation*}\label{eq:Z}
Z=Q +\barl_1+\cdots +\barl_m+\barc_1+\cdots +\barc_n.
\end{equation*}
We consider another point $\zeta\sprime\in \ZZZ\spmn$ with
the decomposition 
\begin{equation*}\label{eq:Zsprime}
Z_{\zeta\sprime}=Q_{u\sprime} +\barl\sprime_1+\cdots +\barl\sprime_m+\barc\sprime_1+\cdots +\barc\sprime_n.
\end{equation*}
%
Recall that the involution $\iota_u$ of $Y_u$ acts on  
the orthogonal complement $\Sigma_u$ of $h_u\in \SY_u$ as $-1$.
Let  $g\colon \SY_u \isom  \SY_{u\sprime}$  be an isometryof lattices.
Suppose that $g$ 
maps $h_u$ to $h_{u\sprime}$. 
Then we have 
$g\circ \iota_{u}=\iota_{u\sprime}\circ g$.
Moreover, by definitions~\eqref{eq:Lu}~and~\eqref{eq:Fu}, 
the  isometry $g$  maps $L_u$ to $L_{u\sprime}$ and $F_u$ to $F_{u\sprime}$.
Therefore 
$g$ induces a bijection $P\spmn_u\isom  P\spmn_{u\sprime}$.
%
\par
Theorem~\ref{thm:main} is an immediate consequence of the following:
\begin{theorem}
The following are equivalent:
\begin{enumerate}[{\rm (i)}]
\item $\zeta$ and $\zeta\sprime$ belong to the same connected component of $\ZZZ\spmn$,
\item $\theta(\zeta)$ and $\theta(\zeta\sprime)$ belong to the same connected component of $\PPP\spmn$,
\item there exists an isometry $g\colon \SY_u \isom  \SY_{u\sprime}$ of lattices 
that maps $h_u$ to $h_{u\sprime}$ and such that
the induced bijection $P\spmn_u\isom  P\spmn_{u\sprime}$ maps $\theta(\zeta)$ to $\theta(\zeta\sprime)$, 
and
\item  there exists a homeomorphism $(\P^2, Z_{\zeta})\isom(\P^2, Z_{\zeta\sprime})$.
\end{enumerate} 
\end{theorem}
\begin{proof}
By Proposition~\ref{prop:selfhomeo},
condition ${\rm (iv)}$ is equivalent to the following:
\begin{itemize}
\item[${\rm (iv)}\sprime$]
there exists 
an orientation-preserving homeomorphism $(\P^2, Z_{\zeta})\isom(\P^2, Z_{\zeta\sprime})$.
\end{itemize} 
We will show that ${\rm (i)}$, ${\rm (ii)}$, ${\rm (iii)}$ and ${\rm (iv)}\sprime$
are equivalent.
The implication ${\rm (i)}\Longleftrightarrow{\rm (ii)}$
follows from Lemma~\ref{lem:smooth}, and  ${\rm (i)}\Longrightarrow{\rm (iv)\sprime}$  
follows from Example~\ref{example:alphastar}.
\par
We show ${\rm (iv)\sprime}\Longrightarrow{\rm (iii)}$.
Suppose that 
$\sigma \colon (\P^2, Z_{\zeta})\isom(\P^2, Z_{\zeta\sprime})$
is an orientation-preserving homeomorphism.
We can assume, after renumbering the curves,
that $\sigma$ induces homeomorphisms 
$\barl_i\isom \barl_i\sprime$ and $\barc_j\isom \barc_j\sprime$
that preserve the orientation.
We have 
a homeomorphism
$\sigma_Y\colon Y_u\isom Y_{u\sprime}$
that covers $\sigma$, and $ \sigma_Y$ induces an isometry 
\[
\sigma_{\SY}\colon \SY_u\isom \SY_{u\sprime},
\] 
which maps $h_u$ to $h_{u\sprime}$.
If  $l_i\subset Y_u$  is  a $Y$-lift of  
a bitangent $\barl_i\subset Z_{\zeta}$,
there exists a $Y$-lift $l_i\sprime\subset Y_{u\sprime}$ of the bitangent $\barl_i\sprime\subset Z_{\zeta\sprime}$
such that $\sigma_Y$ induces a homeomorphism $l_i\isom l_i\sprime$ preserving the orientation.
In particular, we have $\sigma_{\SY}([l_i])=[ l_i\sprime]$.
The same holds for
a $Y$-lift  $c_j\subset Y_u$ of a $4$-tangent conic $\barc_j\subset Z_{\zeta}$.
Hence the bijection $P\spmn_{u}\isom P\spmn_{u\sprime}$  induced by  the isometry $\sigma_{\SY}$ maps 
$\theta(\zeta)$ to $\theta(\zeta\sprime)$.
Thus ${\rm (iii)}$ holds.
\par
We show ${\rm (iii)}\Longrightarrow{\rm (ii)}$.
Suppose that ${\rm (iii)}$ holds.
We choose a path $\beta\colon I\to \UUU$ from $u$ to the base-point $b$
and a path $\beta\sprime\colon I\to \UUU$ from $u\sprime$ to $b$,
and consider  the isometries 
\[
\beta_{*}\colon \SY_u\isom \SY_b, \quad
\beta\sprime_{*}\colon \SY_{u\sprime}\isom \SY_b
\]
obtained by the parallel transports along $\beta$ and $\beta\sprime$.
Note that $\beta_{*}(h_u)=h_b$ and $\beta\sprime_{*}(h_{u\sprime})=h_b$.
Hence $\beta\sprime_*\circ g \circ \beta\inv_{*}$ is an element of $\OG(\SY_b, h_b)$.
By Theorem~\ref{thm:monodromy}, there exists a loop $\alpha\colon I\to \UUU$ 
with the base point $b$ such that
\[
\alpha_*=\beta\sprime_*\circ g \circ \beta\inv_{*}.
\]
Therefore the isometry $g\colon \SY_u\isom \SY_{u\sprime}$ is equal to 
the parallel transport $\gamma_{*}$ along the path $\gamma:=\beta^{\prime-1}\alpha \beta$ from $u$ to $u\sprime$.
Let 
\[
\tilgamma\colon I\to \PPP\spmn
\]
 be the lift of $\gamma$ such that $\tilgamma(0)=\theta(\zeta)$.
Since $g=\gamma_*$ maps $\theta(\zeta)$ to $\theta(\zeta\sprime)$,
we see that $\tilgamma(1)=\theta(\zeta\sprime)$.
Therefore $\theta(\zeta)$ and $\theta(\zeta\sprime)$ are in the same connected component of $\PPP\spmn$.
\end{proof}
\section{Geometry of the $K3$ surface $X_u$}\label{sec:familyofconics2}
We investigate the connected families of  $4$-tangent conics
more closely 
for a \emph{general} point $u\in \UUU$.
Our main result of this section is as follows.
\begin{theorem}\label{thm:63}
Suppose that $u\in \UUU$ is general.
Then each connected family  of $4$-tangent conics $\barc$ of $Q_u$
is parameterized by a  rational curve minus $12+6$ points.
A member 
$\barc$ of this family becomes  a $3$-tangent conic
at  each of $12$ punctured points, 
and 
$\barc$  degenerates into a union of two distinct bitangents
at  each of  the remaining $6$ punctured points.
\end{theorem}
For the proof, we add the following easy result to Proposition~\ref{prop:lifts}.
\begin{proposition}\label{prop:lifts2}
Let $\barc$ be a $3$-tangent conic of $Q_u$.
Then 
$\gamma_u^*(\barc)$ is a union of two one-nodal rational  curves.
\qed
\end{proposition}
Recall that the double covering $\eta_u\colon X_u\to Y_u$ 
induces a primitive embedding of lattices $\eta_u\sp*\colon \SY_u(2)\inj \SX_u$.
\begin{proposition}\label{prop:Kondo}
If  $u\in \UUU$ is general, then 
 $\eta_u^*$ is an isomorphism.
\end{proposition}
\begin{proof}
Kondo~\cite{Kondo2000} studied 
the moduli of genus-$3$ curves
by considering the periods of $K3$ surfaces $X$
that are cyclic covers of $\P^2$ of degree $4$ branched along quartic curves $Q\subset \P^2$.
Let $j$ denote the generator  of $\Gal(X/\P^2)\cong \mu_4$ that
acts on $H^{2,0}(X)$ as $\sqrt{-1}$.
Kondo exhibits an action of the cyclic group
$\mu_4$ on the $K3$ lattice
\[
\LL:=E_8^{\oplus 2}\oplus U^{\oplus 3}
\]
that is obtained by a marking $H^2(X, \Z)\cong \LL$.
Let $\LL_{S}$ and $\LL_{T}$ be the kernel of $j^{*2}-1$ and of $j^{*2}+1$  on $\LL$, respectively.
Then $\LL_{S}$ is of rank $8$,  and,
via the marking,  equal  to the image of
the pull-back of $H^2(Y, \Z)(2)$ by the double covering $X\to Y:=X/\gens{j^2}$.
The period $H^{2,0}(X)$ is a point of $\P_*(V_{\sqrt{-1}})$,  
where $V_{\sqrt{-1}}$
is the kernel of $j^{*}-\sqrt{-1}$ on $\LL_{T}\tensor \C$.
We have $\dim \P_*(V_{\sqrt{-1}})=6$.
The result of~\cite{Kondo2000} implies that,  when $Q$ varies,
the point $H^{2,0}(X)$ of $\P_*(V_{\sqrt{-1}})$ sweeps an open subset of 
$\P_*(V_{\sqrt{-1}})$.
\par
We fix a marking $H^2(X_u, \Z)\cong \LL$.
Since $u\in \UUU$ is general,
the period $H^{2,0} (X_u)$ 
is general in $\P_*(V_{\sqrt{-1}})$.
Since $\LL_{T}\tensor\C=V_{\sqrt{-1}}\oplus \overline{V_{\sqrt{-1}}}$, 
the minimal $\Z$-submodule $M$ of $\LL$ such that  $M\tensor\C$ contains $H^{2,0} (X_u)$
is equal to $\LL_{T}$, and hence its orthogonal complement $M\sperp=\SX_{u}$
is equal to $\LL_{S}=\eta_u^* (\SY_u(2))$.
\end{proof}
Let $\Rats (X_u)$ denote the set of rational curves on $X_u$, 
and  $\Ells (X_u)$  the set of elliptic fibrations on $X_u$.
\begin{proposition}\label{prop:Rats}
Suppose that $u\in \UUU$ is general.
Then $\Rats(X_u)$ is equal to the set $\tilL_u:=\set{\eta_u^*(l)}{l\in L_u}$
of $56$ smooth rational curves on $X_u$.
\end{proposition}
This proposition is proved by Proposition~\ref{prop:Kondo} 
and~\cite[Proposition 98]{XRAtlas2021}.
See also~\cite[Remark 99]{XRAtlas2021}.
We give a proof, however, 
because the argument  is also used in the proof of 
Proposition~\ref{prop:Ells} below.
Recall from the proof of Theorem~\ref{thm:barCCC} that,
for  $v\in F_u$,
there exists an elliptic fibration $\phi_v\colon X_u\to \P^1$
such that 
the class of a  fiber of $\phi_v$ is   $\eta_u^*(v)$.
\begin{proposition}\label{prop:Ells}
Suppose that $u\in \UUU$ is general.
Then $v\mapsto \phi_v$ gives a  bijection $F_u \cong \Ells (X_u)$.
Each  fibration $\phi_v$  has no section.
The singular fibers of $\phi_v$ consist  of  
$6$  fibers of type $\typeI_2$ and $12$ fibers of type $\typeI_1$.
\end{proposition}
\begin{proof}[Proof of Propositions~\ref{prop:Rats}~and~\ref{prop:Ells}]
The space 
\[
\set{v \in \SX_u\tensor\R}{\intfX{v,v}>0}
\]
has two connected components.
Let $\PPP_u$ be the connected component 
containing the ample class $\tilh_u$.
For a vector  $v\in \SX_u\tensor\R$ with $\intfX{v, v}<0$,
let $[v]\sperp$ be the hyperplane of $\SX_u\tensor\R$ defined by $\intfX{x, v}=0$,
and we put $(v)\sperp:=[v]\sperp\cap \PPP_u$.
We then put
\[
N_u:=\set{v\in \PPP_u}{\intfX{v, \Gamma}\ge 0\;\;\textrm{for all curves}\;\; \Gamma\subset X_u}.
\]
It is well known that $N_u$ is equal to
\[
\set{v\in \PPP_u}{\intfX{v, \Gamma}\ge 0\;\;\textrm{for all}\;\; \Gamma\in \Rats(X_u)},
\]
and that each  $\Gamma\in \Rats(X_u)$ defines a wall of the cone $N_u$,
that is, $(\Gamma)\sperp\cap N_u$ contains a non-empty open subset of $(\Gamma)\sperp$.
Let  $\barN_u$ be the closure of $N_u$ in $ \SX_u\tensor\R$.
For the proof of  Proposition~\ref{prop:Rats}, it is enough to show that $\barN_u$ is equal to 
\[
\barNsprime_u:=\set{v\in \SX_u\tensor\R}{\intfX{v, \till}\ge 0\;\;\textrm{for all}\;\; \till\in \tilL_u}.
\]
\par
A \emph{face} of the cone $\barNsprime_u$ is a closed subset 
$F$ of $\barNsprime_u$ of the form
$F=V\cap \barNsprime_u$,
where $V$ is an intersection of some of the hyperplanes
$[\till]\sperp$ ($\till\in \tilL_u$)
such that $F$ contains a non-empty open subset of 
$V$.
We say that $V$ is  the \emph{supporting linear subspace} of the face $F$,
and put $\dim F:=\dim V$.
A \emph{ray} is 
a $1$-dimensional face.
For the proof of  $\barN_u=\barNsprime_u$,
it is enough to show that all rays
of $\barNsprime_u$ are contained in $\barN_u$.
We can calculate all the faces $F$ of $\barNsprime_u$ by descending induction on $d:=\dim F$ 
using linear programming method~(see~\cite[Section 2.2]{ShimadaHessian2019}).
The result is as follows.
Suppose that $d\ge 2$.
Then a linear subspace 
\begin{equation}\label{eq:Vcaps}
V=[\till_1]\sperp\cap\dots\cap [\till_k]\sperp
\end{equation}
with $\till_1, \dots, \till_k\in \tilL_u$ is the supporting 
linear subspace of a face $F$ with $\dim F=d$ if and only if 
$k=8-d$ and $\till_1, \dots, \till_k$ are disjoint from each other,
that is,  their dual graph is the Dynkin diagram of type $(8-d) A_1$.
Suppose that $d=1$.
Then a linear subspace $V$ as~\eqref{eq:Vcaps} is 
the supporting linear subspace of a ray $F$ 
if and only if  one of the following  holds:
\begin{itemize}
\item[($7A_1$)] $k=7$ and the dual graph of $\till_1, \dots, \till_7$ 
is the Dynkin diagram of type $7A_1$.
In this case, $F$  is generated by a  vector $v\in \SX_u$
with $\intfX{\tilh_u, v}=6$ and  $\intfX{v, v}=2$.
There exist  exactly $576$ rays of this type.
\item[($6 \widetilde{A}_1$)]  $k=12$ and the dual graph of $\till_1, \dots, \till_{12}$ is
the Dynkin diagram of  type $6 \widetilde{A}_1$, 
where $\widetilde{A}_1$ is  $\wtAtwo$.
In this case, $F$  is generated by a  primitive vector 
$\tilv$ with 
$\intfX{\tilh_u, \tilv}=4$ and $\intfX{\tilv, \tilv}=0$.
There exist  exactly $126$ rays of this type,
and these generators $\tilv$ are equal to $\eta_u^*(v)$ for some $v\in F_u$.
\end{itemize}
In Table~\ref{table:faces}, the numbers of faces of  $\barNsprime_u$ are given.
\par
Suppose that there exists a ray $F$ of $\barNsprime_u$
not contained in $\barN_u$.
Then the generating class $v\in \SX_u$ of $F$ given above  is effective but  not nef.
Let $D$ be an effective divisor of $X_u$  such that $[D]=v$.
Then $D$ contains a smooth rational curve $\Gamma$ 
with $\intfX{\Gamma, v}<0$ as an irreducible component.
Since $\tilh_u$ is ample,
the $(-2)$-vector $r=[\Gamma]$ satisfies 
$\intfX{\tilh_u, r}<\intfX{\tilh_u, v}\le 6$.
We make the set of all $(-2)$-vectors $r\sprime\in \SX_u$
with $\intfX{\tilh_u, r\sprime}=1, \dots, 5$,
and confirm that this set is equal to the set of classes of  $\tilL_u$.
In particular, 
it contains no element $r\sprime$ satisfying  $\intfX{r\sprime, v}<0$.
This contradiction shows 
$\barNsprime_u=\barN_u$, and 
$\Rats(X_u)=\tilL_u$  is proved.
\par
It is well known that 
there exists a bijection between $\Ells (X_u)$ and the set of rays contained 
in $\barN_u\cap \bdr \barPPP_u$.
Hence we have $|\Ells (X_u)|=126$,   and $v\mapsto \phi_v$ gives a bijection 
from $F_u$ to $\Ells (X_u)$.
Therefore, 
as was shown in the proof of Theorem~\ref{thm:barCCC},
every  fiber $f$ of any elliptic fibration $\phi_v$
is a double cover of  a splitting conic of $Q_u$.
The class of $f$ is equal to  $\eta_u^*(v)$.
Since no  element $\till\in \Rats(X_u)$ satisfies $\intfX{f, \till}=1$,
the fibration $\phi_v$ has no section.
Since the dual graph of the set of $\till\in \Rats(X_u)$ with $\intfX{f, \till}=0$  
is of type $6 \widetilde{A}_1$,  
the fibration $\phi_v$ has exactly $6$ reducible fibers,
each of which is either of type $\typeI_2$ or of type $\typeIII$.
If  $\till_{i}, \till_{j}\in \Rats(X_u)$ are in the same  fiber of $\phi_v$,
then they 
satisfy $\intfX{\till_{i}, \till_{j}} =2$ and hence 
$\barl_i:=\gamma_u(\till_{i})$ and $\barl_j:=\gamma_u(\till_{j})$ are
distinct bitangents of $Q_u$ by Table~\ref{table:pullbacks}.
Since $u\in \UUU$ is general, the intersection point 
of $\barl_i$ and $\barl_j$ is not on $Q_u$. 
Hence  every reducible fiber of $\phi_v$ is of type $\typeI_2$.
The  irreducible singular fibers are either 
of type $\typeI_1$ or of type $\typeII$.
By Lemma~\ref{lem:2222or422} and Proposition~\ref{prop:lifts2}, 
we see that all irreducible singular fibers must be  of type $\typeI_1$.
Calculating  the Euler number,
we conclude that 
the number of singular fibers of type $\typeI_1$ is $12$. 
\end{proof}
\begin{table}
\[
\begin{array}{c | lllllll}
\dim F\;  &\;\;7 & 6 & 5 & 4 & 3 & 2 & 1 \\
\hline 
\# &\;\; 56 & 756 & 4032 & 10080 & 12096 & 6048 & 576+126
\end{array}
\]
%
\caption{Numbers of faces $F$}\label{table:faces}
\end{table}
\begin{remark}\label{rem:576}
The set of $576$ rays  of type $7A_1$ is in bijective correspondence with
the set  $\disjsevenset_u$ in the proof of Theorem~\ref{thm:monodromy}.
A ray $F$ of type $7A_1$ corresponds to a $7$-tuple 
$\{l_1, \dots, l_7\}\in \disjsevenset_u$
as follows.
The generator $v$ of $F$ with $\intfX{v, v}=2$ is the class of
the pull-back of a line of a plane $\PP$ by the double covering
$X_u\to Y_u\to \PP$,
where $Y_u\to \PP$ is the blowing down of  the $(-1)$-curves $l_1, \dots, l_7$.
\end{remark}
\begin{proof}[Proof of Theorem~\ref{thm:63}]
In fact, the proof was already given in the last paragraph of the proof of Proposition~\ref{prop:Ells}.
\end{proof}
\section{Configurations of $Y$-lifts}\label{sec:config}
Throughout this section, let $u$ be a general point of $\UUU$.
\subsection{Lemmas on quartic polynomials}
Let $[d_1, \dots, d_m]$ be a list of positive integers satisfying $d_1+\dots+d_m=4$.
We put
\[
\Gamma(d_1, \dots, d_m\colon 2):= \Gamma(d_1)\times  \dots\times \Gamma( d_m)\times \Gamma( 2),
\]
and denote by $\psi_{[d_1, \dots, d_m]}\colon  \Gamma(d_1, \dots, d_m\colon 2) \to \Gamma(4)$ the morphism
\[
( f_1, \cdots,  f_m, q)\mapsto f_1\cdots f_m+q^2.
\]
\begin{lemma}\label{lem:dominant}
The morphism $\psi_{[d_1, \dots, d_m]}$ is dominant.
\end{lemma}
%
%
\begin{proof}
It is enough to show that  $\psi_{[1,1,1,1]}$ is dominant, and  then, 
it suffices to find a point $P$ of  $\Gamma(1,1,1,1\colon 2)$
at which  the differential of $\psi:=\psi_{[1,1,1,1]}$ is of rank $\dim \Gamma(4)=15$.
By choosing points $P$ randomly and calculating 
the rank of $d_P \psi$, we can easily find such a point.
\end{proof}
\begin{definition}\label{def:Psi}
For $[d_1, \dots, d_m]$ with  $d_1+\dots+d_m=4$,
we have an open dense subset $\VVV_{[d_1, \dots, d_m]}\subset \Gamma(d_1, \dots, d_m\colon 2)$ and a dominant  morphism
\[
\Psi_{[d_1, \dots, d_m]}\colon  \VVV_{[d_1, \dots, d_m]} \to\UUU
\]
such that, for $p=(f_1,\dots, f_m, q)\in \VVV_{[d_1, \dots, d_m]} $,
the quartic curve corresponding $ \Psi_{[d_1, \dots, d_m]} (p)\in \UUU$
is defined by $f_1\cdots f_m+q^2=0$.
\end{definition}
\begin{lemma}\label{lem:q^2}
If $Q_u$ is defined by $f+q^2=0$ with $f\in \Gamma(4)$ and $q\in \Gamma(2)$,
then $Y_u$ has a divisor 
that is mapped isomorphically to the divisor $\{f=0\}$ of $\P^2$.
\end{lemma}
\begin{proof}
The surface $Y_u$ is defined by $w^2=f+q^2$,
where $w$ is a new variable,  and hence 
contains  a divisor defined by 
$f=w-q=0$.
It is obvious that $\pi_u$ maps this divisor 
to the divisor $\{f=0\}$ of $\P^2$ isomorphically.
\end{proof}
\subsection{Triangles of bitangents}\label{subsec:triangles}
Recall that $L_u$ is the set of $Y$-lifts $l$ of bitangents $\barl\in \barL_u$ of $Q_u$.
\begin{definition}
A \emph{triangle on $Y_u$} is a subset $\{l_1, l_2, l_3\}$  of $L_u$
such that $\intfY{l_1, l_2}=\intfY{l_2, l_3}=\intfY{l_3, l_1}=1$.
A \emph{liftable triangle} of bitangents of $Q_u$ is a subset $\{\barl_1, \barl_2, \barl_3\}$  of $\barL_u$
that is the  image of a triangle on $Y_u$ by $\pi_u$.
\end{definition}
Let $\barl_1, \barl_2, \barl_3$ be bitangents of $Q_u$.
We choose $Y$-lifts $l_1, l_2, l_3 \in L_u$
of  $\barl_1, \barl_2, \barl_3$ in such a way that 
$\intfY{l_1, l_2}=\intfY{l_2, l_3}=1$.
Then 
$\{\barl_1, \barl_2, \barl_3\}$ is  liftable  if and only if $\intfY{l_3, l_1}=1$.
\par
Let $T_u$ be 
the set  of triangles on $Y_u$.
We have calculated $L_u\subset \SY_u$ explicitly.
Using this data, we enumerate 
$T_u$, and see that 
$|T_u|=2520$.
Let 
\[
\overline{T}_u:=T_u/\gens{\iota_u}
\]
be the set  of liftable triangles of bitangents of  $Q_u$.
\begin{corollary}\label{cor:1260}
There exist exactly $|\overline{T}_u|=1260$ liftable triangles.
\qed
\end{corollary}
By Theorem~\ref{thm:monodromy},
we obtain the following:
\begin{proposition}\label{prop:2520}
By the monodromy, 
$\pione(\UUU, b)$ acts transitively on $T_b$ and hence on $\overline{T}_b$.
\qed
\end{proposition}
\begin{proposition}\label{prop:f1f2f3f4}
Let $\barl_1, \barl_2, \barl_3$ be bitangents of $Q_u$.
Suppose that $\barl_i$ is defined by $f_i=0$ for $i=1, \dots, 3$, where $f_i\in \Gamma(1)$.
Then $\{\barl_1, \barl_2, \barl_3\}$ is  liftable  if and only if
there exist polynomials $f_4\in \Gamma(1)$ and $q\in \Gamma(2)$ 
such that $Q_u$ is defined by $f_1f_2 f_3 f_4+q^2=0$.
\end{proposition}
\begin{proof}
The if-part follows from Lemma~\ref{lem:q^2}.
Let $\bar{\tau}\colon \barTTT\to \UUU$ be the finite \'etale covering 
obtained from the family 
$\set{\overline{T}_u}{u\in \UUU}$.
Then $\barTTT$ is irreducible
by Proposition~\ref{prop:2520}.
Let $p:=(f_1\sprime, \dots, f_4\sprime, q\sprime)$
be a point of $\VVV_{[1,1,1,1]}$, and we put 
\[
u\sprime:=\Psi_{[1,1,1,1]}(p) \;\in\; \UUU
\]
where $\VVV_{[1,1,1,1]}$ and $\Psi_{[1,1,1,1]}$ are given in Definition~\ref{def:Psi}.
Let  $\barl\sprime_i\subset \P^2$ be  the line  $\{f_i\sprime=0\}$.
By the if-part,
we have $\{\barl\sprime_1, \barl\sprime_2, \barl\sprime_3\}\in \overline{T}_{u\sprime}$.
By $p\mapsto \{\barl\sprime_1, \barl\sprime_2, \barl\sprime_3\}$, 
we obtain a morphism $\Psi_{\barTTT}\colon\VVV_{[1,1,1,1]}\to \barTTT$.
Since $\bar{\tau}\circ \Psi_{\barTTT}=\Psi_{[1,1,1,1]}$, $\bar{\tau}$ is \'etale,  $ \barTTT$ is irreducible, 
and $\Psi_{[1,1,1,1]}$ is dominant, 
we conclude that $\Psi_{\barTTT}$ is dominant.
Since $u\in \UUU$ is  general,
we obtain the proof.
\end{proof}
\begin{corollary}\label{cor:tetra}
There exists  a set  $ \barR_u$ consisting of 
$315$ subsets $\{\barl_a, \barl_b, \barl_c, \barl_d\}\subset \barL_u$ 
of size $4$ with the following properties:
a subset  $\{\barl_i, \barl_j, \barl_k\}\subset \barL_u $ of size $3$ is liftable
if and only if there exists an element $\{\barl_a, \barl_b, \barl_c, \barl_d\}\in  \barR_u$  
containing $\{\barl_i, \barl_j, \barl_k\}$.
\qed
\end{corollary}
\subsection{Pairs of  splitting conics}\label{subsec:cc}
Recall  that 
$F_u\subset \SY_u$ is the set of classes $[c]$ of 
$Y$-lifts $c$ of  $4$-tangent conics $\barc$ of $Q_u$,
and that $\barF_u=F_u/\gens{\iota_u}$
is regarded 
as the set of connected families of $4$-tangent conics of $Q_u$,
or equivalently as  the set of connected families of splitting conics of $Q_u$.
For a splitting conic $\barc$,
let $[\barc]\in \barF_u$ denote 
the connected family containing $\barc$.
By Theorem~\ref{thm:monodromy}, we obtain the following:
\begin{proposition}\label{prop:Ftransitive}
By the monodromy, 
$\pione(\UUU, b)$ acts transitively on $F_b$ and hence on $\barF_b$.
\qed
\end{proposition}
\begin{definition}
Let $\barc$ be a splitting conic of $Q_u$.
We say that 
a decomposition   $\pi_u^*(\barc)=c+c\sprime$
is  \emph{normal}
if  each of $c$ and $c\sprime$ is a $Y$-lift of $\barc$.
\end{definition}
Note that, if $\barc$ is smooth, then the  decomposition $\pi_u^*(\barc)=c+c\sprime$ is  normal,
whereas 
if $\barc$ is a sum  of two bitangents $\barl+\barl\sprime$,
then $\pi_u^*(\barc)=c+c\sprime$ being  normal means that 
$c=l+l\sprime$ with $\intfY{l, l\sprime}=1$.
\begin{definition}
Let $\barc_1$ and $\barc_2$  be  splitting conics of $Q_u$, and 
let    $\pi_u^*(\barc_1)=c_1+c_1\sprime$ and $\pi_u^*(\barc_2)=c_2+c_2\sprime$
be the normal decompositions.
We put
\[
I([\barc_1], [\barc_2]):= \left[
\begin{array}{cc} 
\intfY{c_1, c_2} &  \intfY{c_1, c_2\sprime } \\ 
\intfY{c_1\sprime, c_2} &  \intfY{c_1\sprime, c_2\sprime } 
\end{array}
\right].
\]
Since we can make switchings  $c_1 \leftrightarrow c_1\sprime$ and 
 $c_2  \leftrightarrow c_2\sprime$, 
the matrix  $I([\barc_1], [\barc_2])$ is well-defined only up to
the transpositions of the two rows and of the two columns.
\end{definition}
%
We have calculated $F_u\subset S_u$ explicitly.
Using this data, we see that 
the matrix $I([\barc_1], [\barc_2])$ is one of the following:
\begin{eqnarray*}
I_A &:=& \left[\begin{array}{cc} 0 & 4 \\ 4 & 0 \end{array}\right] \textrm{or} \left[\begin{array}{cc} 4 & 0 \\ 0 & 4 \end{array}\right], \\
I_B &:=& \left[\begin{array}{cc} 2 & 2 \\ 2 & 2 \end{array}\right], \\
I_C &:=& \left[\begin{array}{cc} 1 & 3 \\ 3 & 1 \end{array}\right] \textrm{or} \left[\begin{array}{cc} 3 & 1 \\ 1 & 3 \end{array}\right].\\
\end{eqnarray*}
%
%
%
\begin{proposition}\label{prop:c1c2}
Let $\barc_1=\{g_1=0\}$ and $\barc_2=\{g_2=0\}$  be  splitting conics of  $Q_u$.
Consider the following conditions:
\begin{enumerate}[{\rm (i)}]
\item $[\barc_1]=[\barc_2]$, that is, 
$\barc_1$ and $\barc_2$  belong to the  same connected family.
\item The matrix $I([\barc_1], [\barc_2])$ is equal to $I_A$.
\item There exists a polynomial  $q\in \Gamma(2)$ such that $Q_u$ is defined by 
$g_1g_2+q^2=0$.
\end{enumerate}
Then we have  ${\rm (iii)}\Longrightarrow {\rm (ii)}\Longleftrightarrow {\rm (i)}$.
If ${\rm (i)}$ holds and $\barc_1$ and $\barc_2$ are general in the connected family 
$[\barc_1]=[\barc_2]\in \barF_u$
of splitting conics, then ${\rm (iii)}$ holds.
\end{proposition}
\begin{proof}
The implication ${\rm (i)}\Longrightarrow {\rm (ii)}$ follows immediately from Table~\ref{table:pullbacks},
and the implication ${\rm (iii)}\Longrightarrow {\rm (ii)}$ follows  from Lemma~\ref{lem:q^2}.
Suppose that ${\rm (ii)}$ holds.
Let $\pi_u^*(\barc_1)=c_1+c_1\sprime$ and $\pi_u^*(\barc_2)=c_2+c_2\sprime$ be the normal decompositions.
Interchanging  $c_2$ and $c_2\sprime$  if necessary,
we can assume that   $\intfY{c_1, c_2}=0$.
We put $f_1:=\eta_u^*(c_1)$ and $f_2:=\eta_u^*(c_2)$.
Note that $f_1$ is a fiber of the elliptic fibration $\phi_1\in \Ells(X_u)$
corresponding to the class $ [c_1]\in F_u$ of $c_1$ by $ F_u\cong \Ells(X_u)$.
Since $\intfX{f_1, f_2}=2\intfY{c_1, c_2}=0$,
we conclude that $f_2$ is a fiber of $\phi_1$,
that is, 
the  elliptic fibration 
corresponding to $ [c_2]\in F_u\cong \Ells(X_u)$   is equal to $\phi_1$.
Therefore $\barc_1$ and $\barc_2$ belong to the same connected  family of splitting conics, 
and ${\rm (i)}$ holds.
Thus ${\rm (iii)}\Longrightarrow {\rm (ii)}\Longleftrightarrow {\rm (i)}$ is proved.
\par
Suppose that $[\barc_1]=[\barc_2]$.
Let $\sigma\colon \barFFF\to \UUU$ be the finite \'etale covering  
defined by the family $\set{\barF_u}{u\in \UUU}$.
By Proposition~\ref{prop:Ftransitive},
we see that $\barFFF$ is irreducible.
Let $p:=(g_1\sprime, g_2\sprime, q\sprime)$
be a point of $\VVV_{[2,2]}$,
and we put $u\sprime:=\Psi_{[2,2]} (p)\in \UUU$; 
\[
Q_{u\sprime}=\{g_1\sprime g_2\sprime+q^{\prime\, 2}=0\}.
\]
Let  $\barc\sprime_i$ be  the splitting conic  $\{g_i\sprime=0\}$ of $Q_{u\sprime}$ for $i=1,2$.
By the implication ${\rm (iii)}\Longrightarrow {\rm (i)}$,
we have $[\barc_1\sprime]=[\barc_2\sprime]$ in $\barF_{u\sprime}$.
By $p\mapsto [\barc_1\sprime]$, 
we obtain 
a morphism  $\Psi_{\barFFF}\colon \VVV_{[2,2]}\to \barFFF$.
By the same argument as in the proof of Proposition~\ref{prop:f1f2f3f4}, 
we see that $\Psi_{\barFFF}$ is dominant.
Since $u$ is general in $\UUU$,
the point $(u, [\barc_1])=(u, [\barc_2])$ is general in $\barFFF$
and the fiber $W$ of $\Psi_{\barFFF}$ over $(u, [\barc_1])$ is of dimension
\[
\dim \Gamma(2,2\colon 2) -\dim \UUU=18-14=4.
\]
Let $S:=\set{\barc(t)}{t\in \P^1}$ be the connected family of splitting conics
containing $\barc_1$ and $\barc_2$.
If $(g_1\sprime, g_2\sprime, q\sprime)$ is a point of the fiber $W$, then 
we have two members $\barc\sprime_1=\{g_1\sprime=0\}$ and  $\barc\sprime_2=\{g_2\sprime=0\}$ 
of $S$,
and thus we have a morphism $W\to \P^1\times \P^1$,
where $\P^1$ is the base curve  of the family $S$.
If two points 
$(g_1\sprime, g_2\sprime, q\sprime)$ and
$(g_1\spprime, g_2\spprime, q\spprime)$ of $W$ 
are mapped to the same point of $\P^1\times \P^1$, 
then there exist scalars 
$\lambda_1, \lambda_2\in \C\sptimes$
such that 
$g_1\spprime=\lambda_1 g_1\sprime$ and 
$g_2\spprime=\lambda_2 g_2\sprime$.
By the dimension reason, we see that $W\to \P^1\times \P^1$ is dominant.
Hence,  if $\barc_1=\{g_1=0\}$ and $ \barc_2=\{g_2=0\}$ are general members of the family $S$,
there exists  a polynomial $q\in \Gamma(2)$ such that 
$(g_1, g_2, q)\in W$, that is, $Q_u$ is defined by $g_1g_2+q^2=0$.
\end{proof}
The following two propositions are confirmed by direct computation.
%
\begin{proposition}
Among the $1953$ non-ordered pairs $\{[\barc_1], [\barc_2]\}$ of distinct elements $[\barc_1], [\barc_2]$ of $\barF_u$, 
exactly $945$ pairs  satisfy $I([\barc_1], [\barc_2])=I_B$, and the remaining 
$1008$   pairs  satisfy  $I([\barc_1], [\barc_2])=I_C$.
When $u=b$, 
these two sets of pairs are the orbits of 
the monodromy action of $\pione(\UUU, b)$ on the set of non-ordered pairs of  elements of $\barF_b$.
\qed
\end{proposition}
Recall that each connected family $[c]\in F_u$ of $Y$-lifts of splitting conics 
contains exactly $6$  reducible members,
and the irreducible components $l, l\sprime$ of a reducible member
satisfy $\intfY{l, l\sprime}=1$.
We have a surjective map
\[
\set{\{l, l\sprime\}\;}{\; l, l\sprime\in L_u, \;\intfY{l, l\sprime}=1\;} \to F_u
\]
defined by $\{l,l\sprime\}\mapsto [l]+[l\sprime]$.
Each fiber of size $6$.
The following gives how the cases $I([\barc_1], [\barc_2])=I_B$ and $I([\barc_1], [\barc_2])=I_C$ are distinguished.
\begin{proposition}
Let $[c_1]$ and $[c_2]$ be elements of $F_u$,
and let $[\barc_1]$ and $[\barc_2]$ be their images by  $F_u\to \barF_u$.
Then $I([\barc_1], [\barc_2])=I_B$ holds if and only if there exists a triangle 
$\{l_1, l_2, l_3\}$ on $Y_u$ such that
$[c_1]=[l_1]+[l_3]$ and $[c_2]=[l_2]+[l_3]$.
\qed
\end{proposition}
\subsection{Pairs of a bitangent and a splitting conic}\label{subsec:lc}
Let $\barl$ be a bitangent of $Q_u$ with $\pi_u^*(\barl)=l+l\sprime$,
and let $\barc$ be a splitting conic of $Q_u$ with the normal decomposition 
$\pi_u^*(\barc)=c+c\sprime$.
We put
\[
J(\barl, [\barc]):= \left[
\begin{array}{cc} 
\intfY{l, c} &  \intfY{l, c\sprime } \\ 
\intfY{l\sprime, c} &  \intfY{l\sprime, c\sprime } 
\end{array}
\right].
\]
The matrix $J(\barl, [\barc])$ is one of the following:
\begin{eqnarray*}
J_{\alpha} &:=& \left[\begin{array}{cc} 0 & 2 \\ 2& 0 \end{array}\right] \textrm{or} \left[\begin{array}{cc} 2 & 0 \\ 0 & 2 \end{array}\right], \\
J_{\beta} &:=& \left[\begin{array}{cc} 1 & 1 \\ 1 & 1 \end{array}\right].
\end{eqnarray*}
By direct computation,
we confirm the following:
\begin{proposition}
Let $\barl$ be a bitangent of $Q_u$,
and $\barc$ a splitting conic of $Q_u$.
Then $J(\barl, [\barc])$ is equal to $J_{\alpha}$ 
if and only if the connected family $[\barc]\in \barF_u$ 
of splitting conics 
 has 
a singular member containing $\barl$ as an irreducible component.
\par
When $u=b$, 
the monodromy action of $\pione (\UUU, b)$ acts on the set of pairs $(\barl, [\barc])\in \barL_b\times\barF_b$ 
with $J(\barl, [\barc])=J_{\alpha}$ transitively,
and the set of pairs $(\barl, [\barc])$ 
with $J(\barl, [\barc])=J_{\beta}$ also  transitively.
\qed
\end{proposition}
\section{Intersection graph}\label{sec:intgraph}
\begin{definition}
An \emph{intersection graph} is a pentad
$(V_{\barl}, V_{\barc}, T, E_{\barc\barc}, E_{\barl\barc})$
such that
\begin{itemize}
\item $V_{\barl}$ and  $V_{\barc}$ are finite sets, 
\item $T$ is a subset of $S^3_0(V_{\barl})$, 
\item $E_{\barc\barc}$ is a map $S^2(V_{\barc})\to \{A, B, C\}$, and 
\item $E_{\barl\barc}$ is a map $V_{\barl} \times V_{\barc} \to \{\alpha, \beta\}$.
\end{itemize}
Two intersection graphs
$(V_{\barl}, V_{\barc}, T, E_{\barc\barc}, E_{\barl\barc})$
and
$(V\sprime_{\barl}, V\sprime_{\barc}, T\sprime, E\sprime_{\barc\barc}, E\sprime_{\barl\barc})$
are \emph{isomorphic} if there exists a pair of  bijections
$V_{\barl}\cong V\sprime_{\barl}$ and 
$V_{\barc}\cong V\sprime_{\barc}$ that induces  
$T\cong T\sprime$, 
$E_{\barc\barc}\cong E\sprime_{\barc\barc}$, 
and 
$E_{\barl\barc}\cong E\sprime_{\barl\barc}$.
\end{definition}
\begin{definition}
For a $\Qmn$-curve $Z$ as in~\eqref{eq:Z}, 
we define an intersection  graph 
\[
g(Z):=(V_{\barl}, V_{\barc}, T, E_{\barc\barc}, E_{\barl\barc})
\]
by the following:
\begin{itemize}
\item $V_{\barl}$ is $\{\barl_1, \dots, \barl_m\}$ and $V_{\barc}$ is $\{\barc_1, \dots, \barc_n\}$,
\item $T$ is the set of liftable triangles 
$\{\barl_i, \barl_j, \barl_k\}\subset \{\barl_1, \dots, \barl_m\}$,
\item $E_{\barc\barc}(\barc_i, \barc_j)$ is the type of the matrix 
$I([\barc_i], [\barc_j])$ defined in Section~\ref{subsec:cc}, and
\item $E_{\barl\barc}(\barl_i, \barc_j)$ is the type of the matrix 
$J(\barl_i, [\barc_j])$ defined in Section~\ref{subsec:lc}. 
\end{itemize}
\end{definition}
\begin{remark}\label{rem:equivrelA}
By Proposition~\ref{prop:c1c2},
the relation 
\[
\barc_i\sim\barc_j \Longleftrightarrow E_{\barc\barc}(\barc_i, \barc_j)=A
\]
is an equivalence relation on $V_{\barc}$,
and the functions $E_{\barc\barc}$ and $E_{\barl\barc}$ are compatible with this  equivalence relation.
\end{remark}
\begin{remark}
When $n=0$, the  intersection graph equal to the two-graph  in~\cite{BannaiOhno2020}.
\end{remark}
It is obvious that, if $\zeta$ and $\zeta\sprime$ are in the same connected component 
of $\ZZZ\spmn$, 
the intersection  graphs $g(Z_\zeta)$ and $g(Z_{\zeta\sprime})$ are isomorphic.
The converse is not true in general,
as examples in  the next section show.
\section{Examples}\label{sec:examples}
\subsection{The case $(m, n)=(6, 0)$}\label{subsec:m6n0}
We have  $|P_b^{(6, 0)}|=376740$. 
The action of $W(E_7)$  decomposes  $P_b^{(6, 0)}$ into orbits 
as in Table~\ref{table:Pbm6n0}.
For each orbit $o_i\subset P_b^{(6, 0)}$,
we choose a point $\zeta\in o_i$ and 
indicate the following data of the intersection graph $g(Z_{\zeta})$
of $Z_{\zeta}=Q_u+\barl_1+\dots+\barl_6$:
$|T|=k$ is the number    of the liftable triangles $t_1, \dots, t_k$ in $\{\barl_1, \dots, \barl_6\}$,
and $a_{\nu}$ is the number of 
pairs $\{t_i, t_j\}$ of liftable triangles such that $|t_i\cap t_j|=\nu$.
The orbit $o_1$ and $o_2$ cannot be distinguished by the two-graph $(V_{\barl}, T)$,
but they belong to different $W(E_7)$-orbits,
and hence the corresponding $\Qtype^{(6,0)}$-curves are of different
homeomorphism types. 
\begin{table}
{\small
\[
\begin{array}{c|ccccc}
i & |o_i| & |T| &  a_0 & a_1 & a_2  \\
\hline
1 & 2016 & 0 & 0 & 0 & 0 \\ 
2 & 1008 & 0 & 0 & 0 & 0 \\ 
3 & 30240 & 4 & 0 & 0 & 6 \\ 
4 & 60480 & 6 & 0 & 6 & 9 \\ 
5 & 22680 & 8 & 2 & 10 & 16 \\ 
6 & 181440 & 8 & 2 & 14 & 12 \\ 
7 & 5040 & 8 & 4 & 12 & 12 \\ 
8 & 12096 & 10 & 0 & 30 & 15 \\ 
9 & 60480 & 10 & 2 & 24 & 19 \\ 
10 & 1260 & 12 & 6 & 30 & 30 \\ 
\end{array}
\]
%
%
}
\caption{The orbit decomposition for $(m, n)=(6, 0)$}\label{table:Pbm6n0}
\end{table}
\subsection{The case $n=0$}\label{subsec:n0}
We continue to consider the case where $n=0$.
From the two-graph $g=(V_{\barl}, T)$,
we can construct a graph $\tilde{g}$ whose set of vertices is $T$ and
whose edge connecting $t_{\mu}, t_{\nu}\in T$ has weight $|t_{\mu}\cap t_{\nu}|$.
If the graphs  $\tilde{g}$ and  $\tilde{g}\sprime$ 
are not isomorphic as graphs with weighted edges,
then the two-graphs $g$ and $g\sprime$ are not isomorphic.
Using this method, we prove the following:
\begin{proposition}
Except for the two orbits $o_1$ and $o_2$ in the case $m=6$ described in Section~\ref{subsec:m6n0},
all $W(E_7)$-orbits  of   $P_b^{(m, 0)}$ are distinguished by their two-graphs.
\qed
\end{proposition}
\begin{example}
Let  $o_1\sprime$ and $o_2\sprime$ be the orbits in $ P_b^{(22, 0)}$ 
containing $22$-tuples  
obtained by taking the complement in $\barL_b$ of  $6$-tuples in 
the orbits $o_1\subset P_b^{(6, 0)}$ and $o_2\subset P^{(6, 0)}$ above, respectively.
Let $g_1\sprime$ and $g_2\sprime$ be the two-graphs of $o_1\sprime$ and $o_2\sprime$.
We have $|T|=600$ for both $g_1\sprime$ and $g_2\sprime$.
The associated graphs $\tilde{g}\sprime_1$ and $\tilde{g}\sprime_2$
with weighted edges are not isomorphic.
The graph $\tilde{g}\sprime_1$ has exactly $8203640$ triples $\{t_{\lambda}, t_{\mu}, t_{\nu}\}$ 
of liftable triangles with weight   $|t_{\lambda}\cap t_{\mu}|=|t_{\mu}\cap t_{\nu}| =|t_{\nu}\cap t_{\lambda}|=0$,
whereas  the number of such triples in $\tilde{g}\sprime_2$ is $8203760$.
\end{example}
%
%
\subsection{The case $(m, n)=(0, 3)$}\label{subsec:m0n3}
By Remark~\ref{rem:equivrelA},
the three edges of the graph $(V_{\barc}, E_{\barc\barc})$  are labelled as in the second column of Table~\ref{table:m0n3}.
The set $P^{(0, 3)}_b$ of size $43680$ is decomposed  into nine $W(E_7)$-orbits 
with  sizes given  in the third column of Table~\ref{table:m0n3}.
\begin{table}
{\small
\[
\begin{array}{c|c|l}
i & \textrm{edge labels}  & \textrm{orbit sizes} \\
\hline 
1 & AAA & 63\\
2 & ABB &1890 \\
3 &ACC & 2016 \\
4 &BBB & 3780+315 \\
5 &BBC & 15120\\
6 &BCC & 15120\\
7 & CCC & 5040 +336
\end{array}
%
\]
}
\caption{The orbit decomposition for $(m, n)=(0,3)$}\label{table:m0n3}
\end{table}
\begin{table}
{\small
\[
\begin{array}{c|c|c|l}
i & E_{\barl\barc} & E_{\barc\barc} & \textrm{orbit sizes} \\
\hline
1 & [ [ \alpha, \alpha ], [ \alpha, \alpha ] ] & A & 3780+378  \\ 
2 & [ [ \alpha, \alpha ], [ \alpha, \alpha ] ] & B & 3780+1890  \\ 
3 & [ [ \alpha, \alpha ], [ \alpha, \alpha ] ] & C & 15120  \\
4 & [ [ \alpha, \alpha ], [ \alpha, \beta ] ] & B & 60480  \\
5 & [ [ \alpha, \alpha ], [ \alpha, \beta ] ] & C & 60480+12096  \\
6 & [ [ \alpha, \alpha ], [ \beta, \beta ] ] & A & 12096  \\
7 & [ [ \alpha, \alpha ], [ \beta, \beta ] ] & B & 30240  \\
8 & [ [ \alpha, \alpha ], [ \beta, \beta ] ] & C & 60480  \\
9 & [ [ \alpha, \beta ], [ \alpha, \beta ] ] & B & 45360+7560  \\
10 & [ [ \alpha, \beta ], [ \alpha, \beta ] ] & C & 30240  \\
11 & [ [ \alpha, \beta ], [ \beta, \alpha ] ] & B & 60480  \\
12 & [ [ \alpha, \beta ], [ \beta, \alpha ] ] & C & 30240+ 6048  \\
13 & [ [ \alpha, \beta ], [ \beta, \beta ] ] & B & 120960  \\
14 & [ [ \alpha, \beta ], [ \beta, \beta ] ] & C & 120960  \\
15 & [ [ \beta, \beta ], [ \beta, \beta ] ] & A & 7560  \\
16 & [ [ \beta, \beta ], [ \beta, \beta ] ] & B & 22680+ 3780  \\
17 & [ [ \beta, \beta ], [ \beta, \beta ] ] & C & 45360  
\end{array}
\]
}
\caption{The orbit decomposition for $(m, n)=(2,2)$}\label{table:m2n2}
\end{table}
\subsection{The case $(m, n)=(2,2)$ }\label{subsec:m2n2}
There exist $17$ intersection graphs indicated in Table~\ref{table:m2n2},
where $E_{\barc\barc}$ is shown  by the type of  $I([\barc_1], [\barc_2])$,
and 
\[
E_{\barl\barc}:=[[J(\barl_1, [\barc_1]), J(\barl_1,  [\barc_2])],[J(\barl_2,  [\barc_1]), J(\barl_2,  [\barc_2])]].
\]
The set $P^{(2, 2)}_b$ of size $762048$ is decomposed  into $23$ orbits by the action of $W(E_7)$,
and their sizes are given  in the $4$th column of Table~\ref{table:m2n2}.
\bibliographystyle{plain}
\bibliography{myrefsQZM}

\end{document}